\title{Dependent Random Choice}
\author{
Jacob Fox\thanks{Department of Mathematics, Princeton, Princeton, NJ
08544. Email: {\tt jacobfox@math.princeton.edu}. Research supported
by an NSF Graduate Research Fellowship and a Princeton Centennial
Fellowship.} \and Benny Sudakov\thanks{Department of Mathematics,
UCLA,  Los Angeles, CA 90095. Email: {\tt bsudakov@math.ucla.edu}.
Research supported in part by NSF CAREER award DMS-0812005 and by
USA-Israeli BSF grant.}}

\documentclass[11pt]{article}
\usepackage{amsfonts,amssymb,amsmath,latexsym}

\newenvironment{proof}
      {\medskip\noindent{\bf Proof.}\hspace{1mm}}
      {\hfill$\Box$\medskip}

\def\qed{\ifvmode\mbox{ }\else\unskip\fi\hskip 1em plus 10fill$\Box$}

\newtheorem{theorem}{Theorem}[section]

\newtheorem{lemma}[theorem]{Lemma}
\newtheorem{proposition}[theorem]{Proposition}
\newtheorem{corollary}[theorem]{Corollary}

\makeatletter
\def\Ddots{\mathinner{\mkern1mu\raise\p@
\vbox{\kern7\p@\hbox{.}}\mkern2mu
\raise4\p@\hbox{.}\mkern2mu\raise7\p@\hbox{.}\mkern1mu}}
\makeatother

\begin{document}
\date{}

\maketitle

\begin{abstract}
We describe a simple and yet surprisingly powerful probabilistic technique which shows how to find in a dense graph a
large subset of vertices in which all (or almost all) small subsets have many common neighbors.
Recently this technique has had several striking applications to Extremal Graph Theory, Ramsey Theory,
Additive Combinatorics, and Combinatorial Geometry. In this survey we discuss some of them.
\end{abstract}

\section{Introduction}

A vast number of problems in Ramsey Theory and Extremal Graph Theory
deal with embedding a small or sparse graph in a dense graph. To
obtain such an embedding, it is sometimes convenient to find in a dense graph a large vertex subset $U$ which is rich in the sense
that all (or almost all) small subsets of $U$ have many common
neighbors. Then one can use this set $U$ and greedily embed the
desired subgraph one vertex at a time. In this paper we discuss an
approach to finding such a rich subset.

This approach is based on a very simple yet surprisingly powerful
technique known as {\em dependent random choice}. Early versions of
this technique were proved and applied by various researchers,
starting with R\"odl, Gowers, Kostochka, and Sudakov (see
\cite{KoRo}, \cite{Go2}, \cite{Su}). The basic method, which is an example of the celebrated Probabilistic Method
(see, e.g., \cite{AlSp}), can be roughly described as follows. Given
a dense graph $G$, we pick a small subset $T$ of vertices uniformly
at random. Then the rich set $U$ is simply the set of common
neighbors of $T$. Intuitively it is clear that if some subset of $G$
has only few common neighbors, it is unlikely that all the members
of the random set $T$ will be chosen among these neighbors. Hence,
we do not expect $U$ to contain any such subset. Although this might sound somewhat vague, we will make it more precise in the next section.

The last ten years have brought several striking applications of
dependent random choice to Extremal Graph Theory, Ramsey Theory, Additive Combinatorics, and
Combinatorial Geometry. There are now a growing number of papers which use this
approach and we think that the time has come to give this topic a systematic treatment.
This is the main goal of our survey. In this paper we will attempt to describe most of the known variants of dependent random choice
and illustrate how they can be used to make progress on a variety of combinatorial problems.
We will usually not give the arguments which lead to the best known results or the
sharpest possible bounds, but rather concentrate on explaining the main ideas which we believe have wide applicability.
Throughout the paper, we systematically omit floor and ceiling signs
whenever they are not crucial for the sake of clarity of presentation. All logarithms are in base $2$.

The choice of topics and examples described in this survey is inevitably biased and is not meant to be comprehensive.
We prove a basic lemma in the next section, and give several quick applications in Section 3. In Section 4 we discuss an example, based on the isoperimetric inequality for binary cubes, which shows certain limitations of dependent random choice. Next we present Gowers' celebrated proof of the Balogh-Szemer\'edi lemma, which is one of the earliest applications of this technique to additive combinatorics. In Sections 6 and 7 we study Ramsey problems for sparse graphs and discuss recent progress on some old conjectures of Burr and Erd\H{o}s. Section 8 contains more variants of dependent random choice which were needed to study embeddings of subdivisions of various graphs into dense graphs. Another twist in the basic approach is presented in Section 9, where we discuss graphs whose edges are in few triangles. The final section of the paper contains more applications of dependent random choice and concluding remarks. These additional applications are discussed only very briefly, since they do not require any new alterations of the basic technique.

\section{Basic Lemma}
For  a vertex $v$ in a graph $G$, let $N(v)$ denote the set of
neighbors of $v$ in $G$. Given a subset $U \subset G$, the {\it
common neighborhood $N(U)$ of $U$} is the set of all vertices of $G$
that are adjacent to $U$, i.e., to {\it every} vertex in $U$.
Sometimes, we might write $N_G(v), N_G(U)$ to stress that the underlying graph is
$G$ when this is not entirely clear from the context.

The following lemma (see, e.g., \cite{KoRo, Su, AlKrSu}) is a typical result proved by dependent random choice.
It demonstrates that every dense graph contains a large vertex subset $U$ such that all
small subsets of $U$ have large
common neighborhood. We discuss applications of this lemma in the next section.

\begin{lemma}\label{firstlemma}
Let $a,d,m,n,r$ be positive integers. Let $G=(V,E)$ be a graph with $|V|=n$ vertices and average degree $d=2|E(G)|/n$. If there is a positive integer $t$ such that $$\frac{d^t}{n^{t-1}}-{n \choose r} \left(\frac{m}{n}\right)^t \geq a,$$ then $G$ contains a subset $U$ of at least $a$ vertices such that every $r$ vertices in $U$ have at least $m$ common neighbors.
\end{lemma}

\begin{proof} Pick a set $T$ of $t$ vertices of $V$ uniformly at random with
repetition. Set $A=N(T)$, and let $X$ denote the cardinality of $A$.
By linearity of expectation,
$$\mathbb{E}[X]=\sum_{v \in V(G)}\left(\frac{|N(v)|}{n}\right)^t
=n^{-t}\sum_{v\in V(G)} |N(v)|^t \geq n^{1-t}\left(\frac{\sum_{v\in
V(G)} |N(v)|}{n} \right)^t=\frac{d^t}{n^{t-1}},$$
where the last inequality is by convexity of the function $f(z)=z^t$.

Let $Y$ denote the random variable counting the number of subsets $S
\subset A$ of size $r$ with fewer than $m$ common neighbors. For a
given such $S$, the probability that it is a
subset of $A$ equals $\left(\frac{|N(S)|}{n}\right)^{t}$. Since there are at most ${n \choose r}$ subsets $S$ of size $r$ for which $|N(S)|<m$, it follows that
$$\mathbb{E}[Y] < {n \choose r}\left(\frac{m}{n}\right)^t.$$

By linearity of expectation,
$$\mathbb{E}[X-Y] \geq \frac{d^t}{n^{t-1}}-{n \choose r}\left(\frac{m}{n}\right)^t \geq a.$$ Hence
there exists a choice of $T$ for which the corresponding set $A=N(T)$ satisfies $X-Y \geq a$. Delete one vertex
from each subset $S$ of $A$ of size $r$ with fewer than $m$ common neighbors. We let $U$ be the remaining subset of $A$.
The set $U$ has at least $X-Y \geq a$ vertices and all subsets of size $r$ have
at least $m$ common neighbors. \end{proof}

\section{A Few Quick Applications}
In this section we present four results which illustrate the application of the basic lemma to various extremal problems.

\subsection{Tur\'an numbers of bipartite graphs}

For a graph $H$ and positive integer $n$, the Tur\'an number $\textrm{ex}(n,H)$ denotes the maximum number of edges
of a graph with $n$ vertices that does not contain $H$ as a subgraph. A fundamental problem in extremal graph theory
is to determine or estimate $\textrm{ex}(n,H)$. Tur\'an \cite{Tu} in 1941 determined these numbers for complete graphs.
Furthermore, the asymptotic behavior of Tur\'an numbers for graphs of chromatic number at least $3$ is given by a
well known result of Erd\H{o}s, Stone, and Simonovits (see, e.g., \cite{Bo}).  For bipartite graphs $H$, however, the situation is
considerably more complicated,
and there are relatively few nontrivial bipartite graphs $H$
for which the order of magnitude of $ex(n,H)$ is known.
The following result of Alon, Krivelevich, and Sudakov \cite{AlKrSu} is best
possible for every fixed $r$ , as shown by the constructions in  \cite{KRS}
and in \cite{ARS}. Although, it can be derived also from an earlier result in \cite{Fu},
the proof using dependent random choice is different and provides somewhat
stronger estimates.

\begin{theorem} If $H=(A \cup B, F)$ is a bipartite graph in which all vertices in $B$ have degree at
most $r$, then $\textrm{ex}(n,H) \leq cn^{2-1/r}$, where $c=c(H)$ depends only on $H$. \end{theorem}

\medskip\noindent{\bf Proof.}\hspace{1mm}
Let $a=|A|$, $b=|B|$, $m=a+b$, $t=r$, and $c=\max\left(a^{1/r},\frac{3(a+b)}{r}\right)$ and suppose
$G$ is a graph with $n$ vertices, and at least $cn^{2-1/r}$ edges. Then the average degree $d$ of $G$ satisfies $d \geq
2cn^{1-1/r}$. Using the definition of $c$ and the fact that $r! \geq (r/e)^r$ it is easy to check that
$$\frac{d^t}{n^{t-1}}-{n \choose r}
\left(\frac{m}{n}\right)^t \geq (2c)^r-\frac{n^r}{r!}\left(\frac{a+b}{n}\right)^r \geq
(2c)^r-\left(\frac{e(a+b)}{r}\right)^r \geq c^r \geq a.$$
Therefore we can use Lemma \ref{firstlemma} (with the parameters $a, d, m, n, r, t$ as above)
to find a vertex subset $U$ of $G$ with $|U|=a$ such that all subsets of $U$ of size $r$
have at least $m=a+b$ common neighbors. Now the following embedding lemma completes the proof of the theorem.

\begin{lemma}\label{embedlemma} Let $H=(A \cup B, F)$ be a bipartite graph in which $|A|=a$, $|B|=b$, and the
vertices in $B$ have degree at most $r$. If $G$ is a graph with a vertex subset $U$ with $|U|=a$ such that all
subsets of $U$ of size $r$ have at least $a+b$ common neighbors, then $H$ is a subgraph of $G$. \end{lemma}

\begin{proof}
 We find an embedding of $H$ in $G$ given by an injective function $f:A \cup B \rightarrow V(G)$. Start by defining
an injection $f:A \rightarrow U$ arbitrarily. Label the vertices of $B$ as $v_1,\ldots,v_{b}$. We embed the
vertices of $B$ in this order one vertex at a time. Suppose that the current vertex to embed is $v_i \in B$. Let $N_i
\subset A$ be those vertices of $H$ adjacent to $v_i$, so $|N_i| \leq r$. Since $f(N_i)$ is a subset of $U$ of
cardinality at most $r$, there are at least $a+b$ vertices adjacent to all vertices in $f(N_i)$. As the total number
of vertices already embedded is less than $a+b$, there is a vertex $w \in V(G)$ which is not yet used in the embedding and is adjacent
to all vertices in $f(N_i)$. Set $f(v_i)=w$. It is immediate from the above description that $f$ provides an
embedding of $H$ as a subgraph of $G$. \end{proof}

\subsection{Embedding a $1$-subdivision of a complete graph}
\label{first1subdividesect}
A {\it topological copy} of a graph $H$ is a graph formed by
replacing edges of $H$ by internally vertex disjoint paths. This is
an important notion in graph theory, e.g., the celebrated theorem of
Kuratowski \cite{Ku} uses it to characterize planar graphs.  In the special
case in which each of the paths replacing the edges of $H$ has exactly $t$ internal vertices, it is called a {\it $t$-subdivision} of $H$.

An old conjecture of Mader and Erd\H{o}s-Hajnal which was proved in \cite{BoTh,KoSz}
says that there is a constant $c$ such that every graph with $n$
vertices and at least $cp^2n$ edges contains a topological copy of
$K_p$. This implies that any $n$-vertex graph with $cn^2$ edges contains a
a topological copy of a complete graph on $\Omega(\sqrt{n})$ vertices. An old
question of Erd\H{o}s \cite{Er} asks whether one can strengthen this statement and
find in every graph $G$ on $n$ vertices with
$c_1n^2$ edges  a $1$-subdivision of a complete graph
with at least $c_2 \sqrt{n}$ vertices for some positive $c_2$ depending on
$c_1$. A positive answer to this question
was given in \cite{AlKrSu}. Here we present a short proof from this paper,
showing the existence of such a subdivision.

\begin{theorem}\label{thm1subdivided} If $G$ is a graph with $n$ vertices and $\epsilon n^2$ edges, then $G$
contains a $1$-subdivision of a complete graph with $a=\epsilon^{3/2}n^{1/2}$ vertices. \end{theorem}

\begin{proof} The average degree $d$ of $G$ is $2\epsilon n$. Let $r=2$, $t=\frac{\log n}{2\log
1/\epsilon}$ and let $m$ be the number of vertices of the $1$-subdivision of the complete graph on $a$ vertices. Since
$m={a \choose 2}+a \leq a^2$ and clearly $\epsilon \leq 1/2$, it is easy to check that
$$\frac{d^t}{n^{t-1}}-{n \choose r}
\left(\frac{m}{n}\right)^t \geq (2\epsilon)^t n-\frac{n^2}{2}\epsilon^{3t}=2^tn^{1/2}-\frac{n^{1/2}}{2} \geq n^{1/2}
\geq a.$$
Therefore we can apply Lemma \ref{firstlemma} with these parameters
to find a vertex subset $U$ of $G$ with $|U|=a$ such that every pair of vertices in $U$ have at
least $m$ common neighbors. Note that a $1$-subdivision of the complete graph on $a$ vertices is a bipartite graph with parts of size $a$ and
$b={a \choose 2}$ such that every vertex in the larger part has degree two. Thus we can now complete the proof of this theorem using Lemma \ref{embedlemma}.
\end{proof}

One can improve this result using a more complicated argument. In \cite{AlKrSu} it was shown that
the graph $G$ as above contains a $1$-subdivision of a complete graph of order $\epsilon n^{1/2}/4$. The power of $\epsilon$ in this result cannot be improved. Also in \cite{FoSu5} we extended this theorem to embedding the $1$-subdivision of any graph with $O(n)$ edges. The proof of these results require some additional ideas and are discussed in Section \ref{1subdivided}.

\subsection{Ramsey number of the cube}

For a graph $H$, the Ramsey number $r(H)$ is the minimum positive integer $N$ such that every $2$-coloring of the
edges of the complete graph on $N$ vertices contains a monochromatic copy of $H$.
Determining or estimating Ramsey numbers is one of the
central problems in combinatorics, see the book Ramsey theory \cite{GRS90} for details.
The $r$-cube $Q_r$ is the
$r$-regular graph with $2^{r}$ vertices whose vertex set consists of all binary vectors $\{0,1\}^r$ and two vertices are adjacent if they differ
in exactly one coordinate. More than thirty years ago, Burr and Erd\H{o}s conjectured that $r(Q_r)$ is linear in the number of vertices of the
$r$-cube. Although this conjecture has drawn a lot of attention, it is still open. Beck \cite{Be} showed that $r(Q_r) \leq 2^{cr^2}$.
His result was improved by Graham et al. \cite{GrRoRu} who proved that $r(Q_r) \leq 8(16r)^{r}$.
Shi \cite{Sh}, using ideas of Kostochka and R\"odl \cite{KoRo}, obtained the first polynomial bound for this problem, showing that
$r(Q_r) \leq 2^{cr+o(r)}$ for some $c\approx 2.618$. A polynomial bound on $r(Q_r)$ follows easily from the basic lemma in Section 2.
Using a refined version of dependent random choice, we will later give a general upper bound on the Ramsey number of bipartite
graphs which improves the exponent to $2^{2r+o(r)}$.

\begin{theorem}
$r(Q_r) \leq 2^{3r}$.
\end{theorem}

\begin{proof} In any two-coloring of the edges of the complete graph on $N=2^{3r}$ vertices, the denser of the two
colors has at least $\frac{1}{2}{N \choose 2} \geq 2^{-7/3}N^2$ edges. Let $G$ be the graph of the densest
color, so the average degree $d$ of $G$ is at least $2^{-4/3}N$. Let $t= \frac{3}{2}r$, $m=2^r$ and $a=2^{r-1}$. We have
$$\frac{d^t}{N^{t-1}}-{N \choose r} \left(\frac{m}{N}\right)^t \geq 2^{-\frac{4}{3}t}N-N^{r-t}m^{t}/r! \geq
2^{r}-1 \geq 2^{r-1}.$$
Therefore, applying Lemma \ref{firstlemma} we find in $G$ a subset
$U$ of size $2^{r-1}$ such that every set of $r$ vertices in $U$ has at least $2^r$ common neighbors. Since $Q_r$ is an
$r$-regular bipartite graph with $2^r$ vertices and parts of size $2^{r-1}$, Lemma \ref{embedlemma} demonstrates that $Q_r$ is a subgraph of
$G$. \end{proof}

\noindent
Note that this proof gives a stronger Tur\'an-type result, showing that any subgraph of density $1/2$ contains $Q_r$.

\subsection{Ramsey-Tur\'an problem for $K_4$-free graphs}

Let $\textrm{\bf{RT}}(n,H,f(n))$ be the maximum number of edges a graph $G$ on $n$ vertices can have which has neither $H$ as a subgraph nor an independent set of size $f(n)$. The problem of determining
these numbers was motivated by the classical Ramsey and Tur\'an theorems, and has
 attracted a lot of attention over the last forty years, see e.g., the survey by Simonovits and S\'os \cite{SiSo}.
One of the most celebrated results in this area states that $$\textrm{\bf{RT}}(n,K_4,o(n))=(1+o(1))\frac{n^2}{8}.$$
 That is, every $K_4$-free graph with $n$ vertices and independence number $o(n)$ has at most
 $(1+o(1))\frac{n^2}{8}$ edges, and this bound is tight. The upper bound was proved by Szemer\'edi \cite{Sz} and the lower
 bound obtained by Bollob\'as and Erd\H{o}s \cite{BoEr}. This result is surprising since it is more plausible to suspect that a
 $K_4$-free graph with independence number $o(n)$ has $o(n^2)$ edges. One of the natural questions brought up in
 \cite{EHSSS} and \cite{SiSo} is whether one can still find a $K_4$-free graph with a quadratic number of edges if the independence number $o(n)$ is
 replaced by a smaller function, say $O(n^{1-\epsilon})$ for some small but fixed $\epsilon>0$.
The answer to this question was given by the second author in \cite{Su}.

\begin{theorem}
Let $f(n)=2^{-\omega\sqrt{\log n}}n$, where $\omega=\omega(n)$ is any function which tends to infinity arbitrarily slowly together with $n$.
Then $$\textrm{\bf{RT}}(n,K_4,f(n)) = o(n^2).$$
\end{theorem}

\begin{proof} Suppose for contradiction that there is a $K_4$-free graph $G$ on $n$ vertices with at least
$2^{-\omega^2/2}n^2=o(n^2)$ edges, and independence number less than $a=f(n)$. The average degree $d$ of
$G$ satisfies $d \geq 2\cdot 2^{-\omega^2/2}n$. Let $r=2$, $m=a$, and $t=\frac{2}{\omega}\sqrt{\log n}$. We
have $$\frac{d^t}{n^{t-1}}-{n \choose r} \left(\frac{m}{n}\right)^t \geq \left(2\cdot
2^{-\omega^2/2}\right)^tn-\frac{n^2}{2}2^{-\omega t\sqrt{\log n}}=2^t2^{-\omega\sqrt{\log n}}n
- 1/2 \geq 2^{-\omega\sqrt{\log n}}n=a.$$ Therefore, applying Lemma \ref{firstlemma}, we can find a subset
  $U$ of size $a$ such that every pair of vertices in $U$ has at least $a$ common neighbors. Since $G$ has
  independence number less than $a$, then $U$ has a pair of adjacent vertices $\{u,v\}$. These vertices
  have at least $a$ common neighbors. No pair of common neighbors of $u$ and $v$ is adjacent as otherwise
  $G$ contains a $K_4$. So the common neighbors of $u$ and $v$ form an independent set of size at least
  $a$, a contradiction. \end{proof}

\section{A dense graph without a rich subset of linear size}

Lemma \ref{firstlemma} shows that every sufficiently dense graph on $n$ vertices contains a large set of vertices $U$
with the useful property that every small subset of $U$ has many common neighbors. In many applications it would be
extremely helpful to have both the size of $U$  and the number of common neighbors be linear in $n$.
Unfortunately, this is not possible, as is shown by the following construction of Kostochka and Sudakov; see \cite{KoSu} for more details.

\begin{proposition}
\label{no-linear}
For infinitely many $n$ there is a graph $G$ on $n$ vertices with at least $n(n-2)/4$ edges such that
any subset of $G$ of linear size contains a pair of vertices with at most $o(n)$ common neighbors.
\end{proposition}

\begin{proof}
Indeed, fix $0<c<1/2$ and let $m$ be a sufficiently large integer. Let $G$ be the graph with vertex set $V=\{0,1\}^m$ in which
two vertices $x,y \in V$ are adjacent if the Hamming distance between $x$ and $y$, which is the number of
coordinates in which they differ, is at most $m/2$. The number of vertices of $G$ is $n=2^m$
and it is easy to check that every vertex in $G$ has degree at least $n/2-1$.
Suppose for contradiction that there is $U \subset V$ with $|U| \geq cn$
and every pair of vertices in $U$ has at least $cn$ common neighbors. A classical result of Kleitman
\cite{Kl} states that if
$$|U| > \sum_{i=0}^t {m \choose i},$$
then $U$ contains two vertices $u_1,u_2$ of
Hamming distance at least $2t+1$. We also need a standard Chernoff estimate (see, e.g., Theorem A.4 in Appendix A of
\cite{AlSp}) saying that
\begin{equation}\label{chernoff}
\sum_{0 \leq i \leq \mu/2-\lambda} {\mu \choose
i}= \sum_{\mu/2+\lambda \leq i \leq \mu} {\mu \choose i}
 \leq 2^{\mu}e^{-2\lambda^2/\mu}.
\end{equation}
Since $|U| \geq cn$, this inequality with
$\mu=m$, $\lambda=m^{5/8}$ and the above result of Kleitman with $t=m/2-\lambda$ shows that there are two vertices $u_1,u_2
\in U$ with distance at least $m-2\lambda=m-2m^{5/8}$.

We next use the fact that $u_1$ and $u_2$ are nearly antipodal to show that they have less than $cn$ common
neighbors, contradicting the assumption that all pairs of vertices in $U$ have $cn$ common neighbors.
First, for intuition, let us see that this holds if $u_1$ and $u_2$ are antipodal, i.e., have Hamming
distance $m$. In this case, a common neighbor of these two vertices has distance exactly $m/2$ from each of
them, and hence there are ${m \choose m/2}=o(n)$ such common neighbors.

The analysis is only a little more complicated in the general case. Let $m-k$ be the Hamming distance of
$u_1$ and $u_2$. Thus $u_1$ and $u_2$ agree on $k \leq 2m^{5/8}$ coordinates and without loss of generality we assume that these are
the first $k$ coordinates. Let $S$ denote the set of vertices which are adjacent to $u_1$ and $u_2$. Let $r=m^{3/8}$ and $S_1$ denote the
set of vertices which agree with $u_1$ and $u_2$ in at least $k/2+r$ of the first $k$ coordinates,
and let $S_2=S \setminus S_1$. The number of vertices in $S_1$ is $$2^{m-k}\sum_{i
\geq k/2+r}{k \choose i} < 2^{m-k}\cdot 2^k e^{-2r^2/k}= ne^{-2r^2/k} = o(n),$$ where the first inequality follows from
estimate (\ref{chernoff}) and the last equality uses that $k \leq 2m^{5/8}$ and $r=m^{3/8}$.

The vertices in $S_2$ have Hamming distance at most $m/2$ from both $u_1$ and $u_2$, and agree with
$u_1$ and $u_2$ in less than $k/2+r$ of the first $k$ coordinates. Therefore, on the remaining
$m-k$ coordinates, the vertices in $S_2$ must agree with $u_1$ in more than $m/2-k/2-r$ coordinates and
with $u_2$ in more than $m/2-k/2-r$ coordinates. Since $u_1$ and $u_2$ differ on these coordinates, we have that on the last $m-k$ coordinates the vertices in $S_2$ should agree with $u_1$ in
$i$ places for some $m/2-k/2-r < i < m/2-k/2+r$. Therefore $S_2$ is at most $2^k$ times the sum of all binomial
coefficients ${m-k \choose i}$ with $|m/2-k/2-i|<r$.
There are at most $2r$ such binomial coefficients, and by Stirling's formula
the largest of them is at most $m^{-1/2}2^{m-k}$. Hence, $|S_2|
\leq 2^k\cdot 2r \cdot m^{-1/2}2^{m-k} =2r m^{-1/2}n= o(n)$. Thus, the number of common neighbors of $u_1$ and $u_2$ is
$|S| \leq |S_1|+|S_2|=o(n)$. \end{proof}

\noindent
Note that this argument can also be used to show that the complement of $G$ in this example also does not have a vertex subset
$U$ of linear size in which each pair of vertices in $U$ have linearly many common neighbors.

Despite the above example, in the next section we present a variant of the basic technique which shows that every
dense graph has a linear size subset in which {\it almost all} pairs have a linear number of
common neighbors. This can be used to establish a very important result in additive combinatorics
known as the Balog-Szemer\'edi-Gowers theorem.

\section{A variant of the basic lemma and additive combinatorics}

Although we proved in the previous section that one can not guarantee in every dense graph a linear subset in which
{\it every} pair of vertices has a linear number of common neighbors, the following lemma shows how to find a linear
subset in which {\it almost every} pair of vertices has this property. The result is stated for bipartite graphs
because in the next section we need this setting for an application to additive combinatorics.
This assumption is of course not essential as it is well known that every graph
contains a bipartite subgraph with at least half of its edges.

\begin{lemma}\label{1lemmaBSG} Let $G$ be a bipartite graph with parts $A$ and $B$ and $e(G)=c|A||B|$ edges. Then for
any $0 < \epsilon < 1$, there is a subset $U \subset A$ such that $|U| \geq c|A|/2$ and at least a
$(1-\epsilon)$-fraction of the ordered pairs of vertices in $U$ have at least $\epsilon c^2|B|/2$ common neighbors in
$B$. \end{lemma}

\begin{proof} Pick a vertex $v \in B$ uniformly at random and let $X$ denote the number of neighbors of $v$.
By the Cauchy-Schwarz inequality, we have
$$\mathbb{E}[X^2] \geq \mathbb{E}[X]^2 =\left(\sum_{a \in A}\frac{|N(a)|}{|B|}\right)^2 =
\big(e(G)/|B|\big)^2 = c^2|A|^2.$$
Let $T=(a_1,a_2)$ be an ordered pair of vertices in $A$. Call it {\it bad} if
$|N(T)| < \epsilon c^2|B|/2$. The probability that $T \subset N(v)$ is $|N(T)|/|B|$,
and therefore for a bad pair it is less than $\epsilon c^2/2$.
Let $Z$ denote the number of bad pairs of vertices in $N(v)$. By linearity of expectation, we have
$$\mathbb{E}[Z] < \frac{\epsilon c^2}{2} \cdot |A|^2 \leq \epsilon\mathbb{E}[X^2]/2,$$ and therefore
$$\mathbb{E}[X^2-Z/\epsilon] = \mathbb{E}[X^2]-\mathbb{E}[Z]/\epsilon \geq \mathbb{E}[X^2]/2 \geq c^2|A|^2/2.$$ Hence,
there is a choice of $v$ such that $X^2-Z/\epsilon \geq c^2|A|^2/2$.
Let $U=N(v)$. Then $Z \leq \epsilon X^2=\epsilon|N(v)|^2=\epsilon|U|^2$, i.e.,
at most an $\epsilon$-fraction of the ordered pairs of vertices of $U$ are bad.
We also have $|N(v)|^2=X^2 \geq
c^2|A|^2/2$. Thus $|U|=|N(v)| \geq c|A|/2$, completing the proof. \end{proof}

\subsection{Balog-Szemer\'edi-Gowers theorem}

An early application of dependent random choice appeared in Gowers' proof \cite{Go2} of Szemer\'edi's theorem on
arithmetic progressions in dense subsets of the integers. One of the important innovations which Gowers introduced
in this work is a new approach which gives much better quantitative bounds for a result of Balog and Szemer\'edi,
which we discuss next. The Balog-Szemer\'edi-Gowers theorem has many applications and is
one of the most important tools in additive combinatorics and number theory.

Let $A$ and $B$ be two sets of integers. Define the {\it sumset} $A+B=\{a+b:a \in A, b\in B\}$. For a bipartite
graph $G$ with parts $A$ and $B$ and edge set $E$, define the {\it partial sumset} $A +_G B = \{a+b:a \in A, b \in
B,(a,b) \in E\}$. In many applications in additive combinatorics, instead of knowing $A+A$ one only has access to a dense
subset of this sum. For example, can we draw a useful conclusion from the fact that $A +_G A$ is small for some dense
graph $G$? It is not difficult to see that in such a case $A+A$ can still be very large. Indeed, take $A$ to be a union of
an arithmetic progression of length $n/2$ and $n/2$ random elements, and let $G$ be the complete graph between the
integers in the arithmetic progression. In this case $|A+_G A|=O(n)$ while $|A+A|=\Omega(n^2)$. However, in this case we
are still able to draw a useful conclusion thanks to the result of Balog and Szemer\'edi \cite{BaSz}.
They proved that if $A$ and $B$ are two sets of size $n$, $G$ has at least $cn^2$ edges and $|A +_G B| \leq Cn$,
then one can find $A' \subset A$ and $B' \subset B$ with $|A'| \geq c'n$, $|B'| \geq c'n$
and $|A' + B'| \leq C'n$, where $c'$ and $C'$ depend on $c$ and $C$ but not on $n$.

The original proof of Balog and Szemer\'edi used the regularity lemma and gave a tower-like dependence between the parameters.
Gowers' approach gives a much better bound, showing that $1/c'$ and $C'$ can be bounded by a constant degree
polynomial in $1/c$ and $C$. Our presentation of the proof of the Balog-Szemer\'edi-Gowers theorem
follows that from \cite{SuSzVu} and is somewhat different from the original proof of
Gowers \cite{Go2}. The heart of the proof is the following graph-theoretic lemma, which is of
independent interest. It first appeared in \cite{SuSzVu}, although a variant was already implicit in the work of
Gowers \cite{Go2}.

\begin{lemma}\label{keygraphlemma} Let $G$ be a bipartite graph with parts $A$ and $B$ of size $n$ and with $cn^2$
edges. Then it contains subsets $A' \subset A$ and $B' \subset B$ of size at least $cn/8$,
such that there are at least $2^{-12}c^5n^2$ paths of length three
between every $a \in A'$ and $b \in B'$.
\end{lemma}

Note that the best result which one can prove for paths of length $1$
is substantially weaker. Indeed, a random bipartite
graph with parts of size $n$ with high probability has edge density roughly $1/2$ and does not contain
a complete bipartite graph with parts of size $2\log n$. Before proving Lemma \ref{keygraphlemma}, we first
show how to use it to obtain a quantitative version of the
Balog-Szemer\'edi-Gowers theorem. Let $A$ and $B$ be two sets of integers of cardinality $n$ such that
$|A+_G B| \leq Cn$ and $G$ has at least $cn^2$ edges.
We show next that there are subsets $A' \subset A$ and $B' \subset B$ each with size at least $cn/8$
such that $|A' + B'| \leq 2^{12}C^3c^{-5}n$.

Let $A' \subset A$ and $B' \subset B$ satisfy the assertion of Lemma \ref{keygraphlemma}. For each $a \in A'$ and $b \in
B'$, consider a path $(a,b',a',b)$ of length three from $a$ to $b$. We have $y=a+b=(a+b')-(a'+b')+(a'+b)=x-x'+x''$,
where $x=a+b'$, $x'=a'+b'$, and $x''=a'+b$ are elements of $X=A +_G B$ since $(a,b')$, $(a',b')$, and $(a',b)$ are edges
of $G$. Thus, every $y \in A'+B'$ can be written as $x-x'+x''$ for at least $2^{-12}c^{5}n^2$ ordered triples
$(x,x',x'')$. On the other hand, $|X| \leq Cn$, so there are at most $C^3n^3$ such triples. This implies that
the number of $y$ is at most $C^3n^3/(2^{-12}c^{5}n^2)=2^{12}C^3c^{-5}n$ and $|A'+B'| \leq 2^{12} C^3 c^{-5} n$.

Lemma \ref{1lemmaBSG} is an essential ingredient in the proof of Lemma \ref{keygraphlemma}. Note that while Lemma
\ref{1lemmaBSG} shows that {\it almost all} pairs in the large subset $U$ have many paths of length two between them,
by allowing paths of longer length, we obtain in Lemma \ref{keygraphlemma} large subsets $A'$ and
$B'$ such that {\it every} $a \in A'$ and $b \in B'$ have many paths of length three between them.

\vspace{0.4cm} \noindent {\bf Proof of Lemma \ref{keygraphlemma}.}\, Let $A_1$ denote the set of vertices in $A$ of
degree at least $cn/2$ and let $c_1=\frac{e(A_1,B)}{|A_1||B|}$ be the edge density between $A_1$ and $B$. There are less
than $(cn/2)\cdot
n=cn^2/2$ edges not containing a vertex in $A_1$ and therefore at least $cn^2/2$ edges between $A_1$ and $B$. Also
$c_1 \geq c$, since deleting vertices in $A \setminus A_1$ increases the edge density of the remaining graph.
In addition, we have $c_1 \geq \frac{cn^2/2}{|A_1||B|}=\frac{cn}{2|A_1|}$. Applying Lemma \ref{1lemmaBSG} to the
induced
subgraph of $G$ with
parts $A_1$ and $B$, with $c$ replaced by $c_1$ and $\epsilon=c/16$, we find a subset $U \subset A_1$ with $|U| \geq
c_1|A_1|/2 \geq cn/4$ such that at most $c|U|^2/16$ ordered pairs of vertices in $U$ are {\em bad}, meaning they have
less than $\epsilon c_1^2 n/2 \geq c^3n/32$ common neighbors in $B$. Let $A'$ denote the set of all vertices $a$ in $U$
that are in at most $c|U|/8$ bad pairs $(a,a')$ with $a' \in U$. The number of bad pairs in $U$ is at least
$\big(c|U|/8\big) \cdot |U \setminus A'|$ and at most $c|U|^2/16$. Thus $|U \setminus A'| \leq |U|/2$ and
$|A'| \geq |U|/2 \geq cn/8.$

Let $B'$ be the set of vertices in $B$ that have at least $c|U|/4$ neighbors in $U$. Recall that every vertex in
$A_1$ and hence in $U$ has degree at least $cn/2$, so there are at least $(cn/2) \cdot |U|$ edges between $U$
and $B$. Since the number of edges between $B \setminus B'$ and $U$ is less than
$\big(c|U|/4\big)\cdot n$,
there are at least $cn|U|/4$ edges between $U$ and $B'$. In particular, since each vertex in $B'$ can have at most
$|U|$ neighbors in $U$, we have $|B'| \geq cn/4$.

Pick arbitrary $a \in A'$ and $b \in B'$. The number of neighbors of $b$ in $U$ is at least
$c|U|/4$. By construction of $A'$, $a$ forms at most $c|U|/8$ bad pairs with other vertices  $a' \in U$.
Hence, there are at least $c|U|/4-c|U|/8-1 \geq c^2n/32-1$ neighbors $a'\not =a$ of $b$ such that $a$ and $a'$ have at
least $c^3n/32$ common neighbors $b'$ in $B$. At least $c^3n/32-1$ of these $b'$ are not $b$. This gives
$\left(c^2n/32-1\right)\left(c^3n/32-1\right) \geq 2^{-12}c^5n^2$ paths $(a,b',a',b)$ of length three from $a$
to $b$ and completes the proof.
\qed

\subsection{An application to an extremal problem}

The ideas which we discussed earlier in this section can also be used to settle an open
problem of Duke, Erd\H{o}s, and R\"odl \cite{DuErRo2} on cycle-connected subgraphs.
Let $\mathcal{H}$ be a collection of graphs. A graph $G$ is {\it
$\mathcal{H}$-connected} if every pair of edges of $G$ are contained in a subgraph $H$ of $G$, where $H$ is a member of
$\mathcal{H}$. For example, if $\mathcal{H}$ is the collection of all paths, then ignoring isolated vertices,
$\mathcal{H}$-connectedness is equivalent to connectedness. If $\mathcal{H}$ consists of all paths of length at most
$d$, then each $\mathcal{H}$-connected graph has diameter at most $d$, while every graph of diameter $d$ is
$\mathcal{H}$-connected for $\mathcal{H}$ being the collection of all paths of length at most $d+2$. So
$\mathcal{H}$-connectedness extends basic notions of connectedness.

The definition of $\mathcal{H}$-connectedness was introduced and studied by Duke, Erd\H{o}s, and R\"odl in the early 1980s. A graph is {\it $C_{2k}$-connected} if it
is $\mathcal{H}$-connected, where $\mathcal{H}$ is the family of even cycles of length at most $2k$. Duke, Erd\H{o}s,
and R\"odl studied the maximum number of edges of a $C_{2k}$-connected subgraph that one can
find in every graph with $n$ vertices and $m$ edges.
In 1984, they asked if there are constants $c,\beta_0>0$ such that
every graph $G$ with $n$ vertices and $n^{2-\beta}$ edges with $\beta < \beta_0$
contains a subgraph $G'$ with at least $cn^{2-2\beta}$ edges in which every two edges
lie together on a cycle of length at most $8$. In \cite{FoSu1}, we answered this question affirmatively,
using a similar version of dependent random choice as in the proof of the Balog-Szemer\'edi-Gowers theorem together with
some additional combinatorial ideas.

A graph is {\it strongly
$C_{2k}$-connected} if it is $C_{2k}$-connected and every pair of edges sharing a vertex lie together on a cycle of
length at most $2k-2$. It is shown in \cite{FoSu1} that for $0<\beta<1/5$ and sufficiently large $n$, every graph $G$ on
$n$ vertices and at least $n^{2-\beta}$ edges has a strongly $C_8$-connected subgraph $G'$ with at least
$\frac{1}{64}n^{2-2\beta}$ edges.
A disjoint union of $n^{\beta}$ complete graphs each of size roughly $n^{1-\beta}$ shows that
this bound on the number of edges of $G'$ is best possible apart from the constant factor.
Also Duke Erd\H{o}s, and R\"odl \cite{DuErRo2} showed that the largest $C_6$-connected subgraph which one can guarantee
has only $cn^{2-3\beta}$ edges.

The following result from \cite{FoSu1} strengthens the key lemma (Lemma \ref{keygraphlemma}), used in the
proof of the Balog-Szemer\'edi-Gowers theorem. It shows that the paths of length three can be taken to lie {\it
entirely} within the subgraph $G'$ of $G$ induced by $A' \cup B'$.
The proof is very close to the proof of the result on strongly $C_8$-connected subgraphs,
discussed above. We wonder if this result might have new applications in additive combinatorics or elsewhere.

\begin{proposition}
Let $G$ be a bipartite graph with parts $A$ and $B$ of large enough size $n$ and with $cn^2$
edges. Then it contains subsets $A' \subset A$ and $B' \subset B$ such that the subgraph
$G'$ of $G$ induced by $A' \cup B'$ has at least $2^{-6}c^2n^2$ edges and at least $2^{-24}c^7n^2$ paths of length three
between every $a \in A'$ and $b \in B'$.
\end{proposition}

\section{Ramsey numbers of bounded degree graphs}
In the previous section we used dependent random choice to show that
every dense graph contains a linear set in which almost every pair has a
linear number of common neighbors. Here we extend
this result from pairs to small subsets and prove a very simple and useful embedding lemma which shows how to use
this linear set to embed sparse bipartite graphs. To illustrate the
power of the combination of these two tools we give an upper bound on
Ramsey numbers of sparse bipartite graphs. Another application of
this technique to Ramsey numbers of bounded degree hypergraphs will be discussed as well.

Recall that for a graph $H$, the {\it Ramsey number} $r(H)$ is the least
positive integer $N$ such that every two-coloring of the edges of the
complete graph $K_N$ on $N$ vertices contains a monochromatic copy of
$H$. Classical results of Erd\H{o}s and Szekeres~\cite{ErSz} and
Erd\H{o}s~\cite{Er3} imply that $2^{k/2} \leq r(K_k) \leq 2^{2k}$ for $k
\geq 2$. Despite extensive efforts by many researchers in the last 60 years, the constant factors in
the above exponents remain the same.

Besides the complete graph, probably the next most classical topic in this area concerns the Ramsey numbers of sparse
graphs, i.e., graphs with certain upper bound constraints on the degrees of the vertices. The study of these numbers was
initiated by Burr and Erd\H{o}s \cite{BuEr} in 1975, and this topic has since played a central role in graph Ramsey theory.

One of the main conjectures of Burr and Erd\H{o}s states that for each
positive integer $\Delta$, there is a constant $c(\Delta)$ such that
every graph $H$ with $n$ vertices and maximum degree $\Delta$ satisfies
$r(H) \leq c(\Delta)n$. This conjecture was proved by Chvat\'al et al.
\cite{ChRoSzTr} using Szemer\'edi's regularity lemma \cite{KoSi}. The
use of this lemma forces the upper bound on $c(\Delta)$ to grow as a
tower of $2$s with height polynomial in $\Delta$. Since then, the
problem of determining the correct order of magnitude of $c(\Delta)$
has received considerable attention from various
researchers. Graham, R\"odl, and Rucinski \cite{GrRoRu} gave the first
linear upper bound on Ramsey numbers of bounded degree graphs without
using any form of the regularity lemma. Their bound was recently improved by a $\log \Delta$ factor in the exponent 
by Conlon and the authors \cite{CoFoSu10} to $c(\Delta)<2^{c\Delta\log \Delta}$.

The case of bounded degree bipartite graphs was studied by Graham,
R\"odl, and Rucinski more thoroughly in \cite{GrRoRu1}, where they
improved their upper bound, showing that $r(H) \leq 2^{\Delta\log \Delta
+O(\Delta)}n$ for every bipartite graph $H$ with $n$ vertices and
maximum degree $\Delta$. In the other
direction, they proved that there is a positive constant $c$ such that,
for every $\Delta \geq 2$ and $n \geq \Delta+1$, there is a bipartite
graph $H$ with $n$ vertices and maximum degree $\Delta$ satisfying $r(H)
\geq 2^{c\Delta}n$. We present the proof from \cite{FoSu5} that the
correct order of magnitude of the Ramsey number of bounded degree
bipartite graphs is essentially given by the lower bound. This is a consequence of the following more general density-type theorem (proved in Section \ref{profmndens}). A similar result with a slightly weaker bound was independently proved by Conlon \cite{Co}.

\begin{theorem}\label{maindensity}
Let $H$ be a bipartite graph with $n$ vertices and maximum degree
$\Delta \geq 1$. If $\epsilon>0$ and $G$ is a graph with $N \geq
8\Delta \epsilon^{-\Delta}n$ vertices and at least $\epsilon{N
\choose 2}$ edges, then $H$ is a subgraph of $G$.
\end{theorem}

Taking $\epsilon=1/2$ together with the majority color in a $2$-edge coloring of $K_N$, we obtain a tight (up to constant factor in the exponent) upper bound on Ramsey numbers of bounded degree bipartite graphs.
It also gives the best known upper bound for the Ramsey number of the $d$-cube.

\begin{corollary}\label{cormain}
If $H$ is bipartite, has $n$ vertices and maximum degree $\Delta
\geq 1$, then $r(H) \leq \Delta2^{\Delta+3}n$. In particular, the Ramsey number of
the $d$-cube $Q_d$ is at most $d2^{2d+3}$.
\end{corollary}

\subsection{Dependent random choice lemma}

A {\it $d$-set} is a set of size $d$. The following extension of Lemma
\ref{1lemmaBSG} shows that every dense graph contains a large
set $U$ of vertices such that almost every $d$-set in $U$ has many
common neighbors.

\begin{lemma}\label{lem:dependent}
If $\epsilon>0$, $d \leq n$ are positive integers, and $G=(V,E)$ is a graph with $N > 4d\epsilon^{-d}n$ vertices and at least $\epsilon N^2/2$ edges, then there is a
vertex subset $U$ with $|U| > 2n$ such that the fraction of $d$-sets $S \subset U$ with $|N(S)| < n$ is less than $(2d)^{-d}$.
\end{lemma}
\begin{proof}
Let $T$ be a subset of $d$ random vertices, chosen
uniformly with repetitions. Set $U=N(T)$, and let $X$ denote the
cardinality of $U$. By linearity of expectation and by convexity of $f(z)=z^d$,
$$\mathbb{E}[X]=\sum_{v \in V}\left(\frac{|N(v)|}{N}\right)^d=
N^{-d}\sum_{v \in V}|N(v)|^d \geq N^{1-d}\left(\frac{\sum_{v \in
V}|N(v)|}{N}\right)^d \geq \epsilon^dN.$$

Let $Y$ denote the random variable counting the number of $d$-sets
in $U$ with fewer than $n$ common neighbors. For a given $d$-set
$S$, the probability that $S$ is a subset of $U$ is
$\left(\frac{|N(S)|}{N}\right)^{d}$. Therefore, we have
$$\mathbb{E}[Y] \leq {N \choose d}\left(\frac{n-1}{N}\right)^d .$$
By convexity, $\mathbb{E}[X^d] \geq
\mathbb{E}[X]^d$. Thus, using linearity of expectation, we obtain
$$\mathbb{E}\left[X^d-\frac{\mathbb{E}[X]^d}{2\mathbb{E}[Y]}\, Y-\frac{\mathbb{E}[X]^d}{2}\right] \geq 0.$$
Therefore, there is a choice of $T$ for which this expression is nonnegative. Then
$$X^d \geq \frac{1}{2}\mathbb{E}[X]^d \geq \frac{1}{2}\epsilon^{d^2}N^d$$
and hence $|U|=X \geq \epsilon^dN/2 > 2n$. Also, $$Y\leq 2X^d\mathbb{E}[Y]\mathbb{E}[X]^{-d}<
2|U|^d{N\choose d} \left(\frac{n}{N}\right)^d\frac{1}{\epsilon^{d^2}N^d}
< \left(\frac{2n}{\epsilon^dN}\right)^d{|U|
\choose d} \leq (2d)^{-d}{|U| \choose d},$$
where we use that $|U|^d \leq 2^{d-1}d!{|U| \choose d}$ which follows from  $|U| > 2n \geq 2d$.
\end{proof}

\subsection{Embedding lemma and the proof of Theorem \ref{maindensity}}\label{profmndens}

Next we show how to embed a sparse bipartite graph in a graph containing
a large vertex set almost all of whose small subsets have many common
neighbors. This will be used to deduce Theorem
\ref{maindensity}.

\begin{lemma}\label{embeddinglemma}
Let $H$ be a bipartite graph on $n$ vertices with maximum degree $d$. If a graph $G$ contains a subset $U$ such that $|U|> 2n$ and
the fraction of $d$-sets in $U$ with less than $n$ common neighbors is less than $(2d)^{-d}$, then $G$ contains a copy of $H$.
\end{lemma}
\begin{proof}
Call a subset $S \subset U$ of size $|S| \leq d$ {\it good} if $S$ is
contained in more than $\big(1-(2d)^{|S|-d}\big){|U| \choose d-|S|}$ $d$-sets in $U$ with at least $n$ common neighbors.
For a good set $S$ with $|S|<d$ and a vertex
$j \in U \setminus S$, call $j$ {\it bad} with respect to $S$ if $S
\cup \{j\}$ is not good. Let $B_S$ denote the
set of vertices $j \in U \setminus S$ that are bad with respect to $S$.
The key observation is that if $S$ is good with $|S|<d$, then
$|B_S| \leq |U|/(2d)$. Indeed, suppose $|B_S|>|U|/(2d)$, then the
number of $d$-sets containing $S$ that have less than $n$ common neighbors in $G$ is at
least
$$\frac{|B_S|}{d-|S|}(2d)^{|S|+1-d}{|U| \choose d-|S|-1} >(2d)^{|S|-d}{|U|
\choose d-|S|},$$ which contradicts the fact that $S$ is good.

Let $V_1$ and $V_2$ denote the two parts of the bipartite graph $H$.
Fix a labeling $\{v_1,\ldots,v_{|V_1|}\}$ of the vertices of $V_1$. Let $L_i=\{v_1,\ldots,v_i\}$.
Since the maximum degree of $H$ is $d$, for every vertex $v_i$,
there are at most $d$ subsets $S \subset L_i$ containing $v_i$ such that $S=N(w) \cap L_i$ for some vertex $w \in V_2$. We use induction on $i$ to find an embedding
$f$ of $V_1$ in $U$ such that for each $w \in V_2$, the set $f(N(w) \cap
L_i)$ is good. Once we have found $f$, we then embed vertices in $V_2$ one by one. Suppose that the current vertex
to embed is $w \in V_2$. Then $f(N(w))=f(N(w) \cap L_{|V_1|})$ is good and hence $f(N(w))$ has at least $n$ common neighbors.
Since less than $n$ of them were so far occupied by other embedded vertices,
we still have an available vertex to embed $w$. We can thus complete the embedding of $H$ in $G$.

It remains to construct the embedding $f$.
By our definition, the empty set is good. Assume at step $i$, for
all $w \in V_2$ the sets $f(N(w) \cap L_i)$ are good. Note that if $w$ is not adjacent
to $v_{i+1}$, then $N(w) \cap L_{i+1}=N(w) \cap L_i$ and therefore
$f(N(w) \cap L_i)$ is good. There are at most $d$
subsets $S$ of $L_{i+1}$ that are of the form $S=N(w) \cap L_{i+1}$
with $w$ a neighbor of $v_{i+1}$. By the
induction hypothesis, for each such subset $S$, the set $f(S
\setminus \{v_{i+1}\})$ is good and therefore there are at most
$\frac{|U|}{2d}$ bad vertices in $U$ with respect to it. In total
this gives at most $d\frac{|U|}{2d}=|U|/2$ vertices. The remaining at
least $|U|/2-i > 0$ vertices in $U \setminus f(L_i)$ are good with
respect to all the above sets $f(S \setminus \{v_{i+1}\})$ and we
can pick any of them to be $f(v_{i+1})$. Notice that this
construction guarantees that $f(N(w) \cap L_{i+1})$ is good for every
$w \in V_2$. In the end of the process, we obtain the desired mapping
$f$ and hence $G$ contains $H$.
\end{proof}

\vspace{0.1cm}
\noindent {\bf Proof of Theorem \ref{maindensity}.}\,
Let $G$ be a graph with $N \geq 8\Delta\epsilon^{-\Delta} n$ vertices and at
least $\epsilon{N \choose 2}=(1-1/N)\epsilon N^2/2$ edges. Let $\epsilon_1=(1-1/N)\epsilon$.
Since $(1-1/N)^{\Delta} > 1/2$, the graph $G$ has at least $\epsilon_1 N^2/2$ edges and $N \geq 4\Delta\epsilon_1^{-\Delta} n$. Thus, by Lemma \ref{lem:dependent} (with $d=\Delta$), it contains a subset $U$ with $|U| > 2n$ such that the fraction of
$\Delta$-sets $S \subset U$ with $|N(S)| < n$ is less than
$(2\Delta)^{-\Delta}$. By Lemma \ref{embeddinglemma} (with $d=\Delta$), $G$ contains every
bipartite graph $H$ on $n$ vertices with maximum degree at most
$\Delta$. \qed

\subsection{Ramsey numbers of sparse hypergraphs}

A hypergraph $H=(V,E)$ consists of a set $V$ of vertices and a set $E$
of subsets of $V$ called edges. A hypergraph is {\it $k$-uniform} if
each edge has
exactly $k$ vertices. The {\it Ramsey number} $r(H)$ of a $k$-uniform
hypergraph $H$ is the smallest integer $N$ such that any $2$-coloring
of the edges of the complete $k$-uniform hypergraph $K_N^{(k)}$ contains
a monochromatic copy of $H$. To understand the
growth of Ramsey numbers for hypergraphs, it is useful to introduce the
tower function $t_i(x)$, which is defined by $t_1(x)=x$ and
$t_{i+1}(x)=2^{t_i(x)}$, i.e., \[t_{i+1}(x)=2^{2^{\Ddots^{2^{x}}}},\]
where the number of $2$s in the tower is $i$. Erd\H{o}s, Hajnal, and
Rado proved (see \cite{GRS90}), for $H$ being the complete $k$-uniform
hypergraph $K_l^{(k)}$, that $t_{k-1}(cl^2) \leq r(H) \leq t_k(c'l)$,
where the constants $c,c'$ depend on $k$.

One can naturally try to extend the sparse graph Ramsey results to
hypergraphs. Kostochka and R\"odl \cite{KR06} showed that for every
$\epsilon>0$, the Ramsey number of any $k$-uniform hypergraph $H$
with $n$ vertices and maximum degree $\Delta$ satisfies $r(H) \leq
c(\Delta, k, \epsilon) n^{1 + \epsilon},$ where
$c(\Delta,k,\epsilon)$ only depends on $\Delta$, $k$, and
$\epsilon$. Since the first proof of the sparse graph Ramsey theorem
used Szemer\'edi's regularity lemma, it was therefore natural to
expect that, given the recent advances in developing a hypergraph
regularity method, linear bounds might also be provable for hypergraphs. Such a program was recently
pursued by several authors \cite{CFKO07, CFKO072,NORS07}, with the result that we now have the following
theorem.

\begin{theorem}
For positive integers $\Delta$ and $k$, there exists a
constant $c(\Delta,k)$ such that if $H$ is a $k$-uniform hypergraph with $n$ vertices and maximum degree $\Delta$, then $r(H) \leq c(\Delta,k)n$.
\end{theorem}

In \cite{CoFoSu}, the authors together with Conlon applied
the tools developed in this section to give a relatively short proof
of the above theorem. From the new proof, it follows that for $k \geq 4$ there is a constant $c$ depending only on $k$
such that $c(\Delta,k) \leq t_k(c\Delta)$. This significantly improves on the Ackermann-type upper bound that arises
from the regularity methods. Moreover, a construction
given in \cite{CoFoSu} shows that, at least in certain cases,
this bound is not far from best possible.

\section{Degenerate graphs}

A graph is {\it $r$-degenerate} if every one of its subgraphs contains a vertex of degree at most $r$.
In particular, graphs with maximum degree $r$ are $r$-degenerate. However, the star on $r+1$ vertices is $1$-degenerate
but has maximum degree $r$. This shows that even $1$-degenerate graphs can have arbitrarily large maximum degree.
The above notion nicely captures the concept of sparse graphs as every $t$-vertex subgraph of a
$r$-degenerate graph has at most $rt$ edges. To see this, remove from the subgraph a vertex of minimum degree,
and repeat this process in the remaining subgraph until it is empty. The number of edges removed at each step is at
most $r$, which gives in total at most $rt$ edges.

In this section we discuss degenerate graphs and
describe a very useful twist on the basic dependent choice approach which is needed to handle embeddings of
such graphs. We then present several applications of this technique to classical extremal problems for degenerate
graphs.

\subsection{Embedding a degenerate bipartite graph in a dense graph}
To discuss embeddings of degenerate graphs, we need first to establish a simple, but very useful, property which these
graphs have. For every $r$-degenerate $n$-vertex graph $H$ there is an ordering of its vertices
$v_1,\ldots,v_n$ such that, for each $1 \leq i \leq n$, the vertex $v_i$ has at most $r$ neighbors $v_j$ with $j<i$.
Indeed, this ordering can be constructed as follows. Let $v_n$ be a vertex of degree at most $r$. Once
we have picked $v_n,\ldots,v_{n-h+1}$, let $v_{n-h}$ be a vertex of
degree at most $r$ in the subgraph of $H$ induced by the not yet labeled vertices. It is easy to see that this ordering
of the vertices has the desired property.

To simplify the presentation we consider only bipartite degenerate graphs. To embed these graphs
into a graph $G$, we find in $G$  two vertex subsets such that every small set in one of them has many common
neighbors in the other and vice versa.

\begin{lemma} Let $G$ be a graph with vertex subsets $U_1$ and $U_2$ such that, for $k=1,2$, every subset of at
most $r$ vertices in $U_k$ have at least $n$ common neighbors in $U_{3-k}$. Then $G$ contains every $r$-degenerate
bipartite graph $H$ with $n$ vertices. \end{lemma}

\begin{proof} Let $v_1,\ldots,v_n$ be an ordering of the vertices of $H$ such that, for $1 \leq i \leq n$, vertex $v_i$
has at most $r$ neighbors $v_j$ with $j<i$. Let $A_1$ and $A_2$ be the two parts of $H$. We find an embedding $f:V(H)
\rightarrow V(G)$ of $H$ in $G$ such that the image of the vertices in $A_k$ belongs to $U_k$ for $k= 1,2$. We embed
the vertices of $H$ one by one, in the above order. Without loss of generality, suppose that the vertex $v_i$ we want to
embed is
in $A_1$. Consider the set $\{f(v_j): j<i,(v_j,v_i) \in E(H)\}$ of images of neighbors of $v_i$ which are already
embedded. Note that this set belongs to $U_2$, has cardinality at most $r$ and therefore has at least $n$ common
neighbors in $U_1$. All these neighbors can be used to embed $v_i$ and at least one of them is yet not occupied, since
so far we embedded less than $n$ vertices. Pick such a neighbor $w$ and set $f(v_i)=w$.
\end{proof}

To find a pair of subsets with the above property, we use a variant of dependent random choice which was first
suggested by Kostochka and Sudakov \cite{KoSu}. Our adaptation of this method is chosen with particular applications in
mind.

\begin{lemma} Let $r,s \geq 2$ and let $G$ be a graph with $N$ vertices and at least $N^{2-1/(s^3r)}$ edges.
Then $G$ contains two subsets $U_1$ and $U_2$ such that, for $k=1,2$, every $r$-tuple in $U_k$
has at least $m=N^{1-1.8/s}$ common neighbors in $U_{3-k}$. \end{lemma}

\begin{proof} Let $q=\frac{7}{4}rs$. Apply Lemma \ref{firstlemma} with $a=N^{1-1/s}$, $d=2e(G)/N
 \geq 2N^{1-1/(s^3r)}$, $n$
replaced by $N$, $r$ replaced by $q$, and $t=s^2r$. We can apply this lemma since
$q=1.75rs<1.8t/s=1.8rs$ and
$$\frac{d^t}{N^{t-1}}-{N \choose q}
\left(\frac{m}{N}\right)^t \geq 2^tN^{1-t/(s^3r)}-\frac{N^q}{q!}N^{-1.8t/s} \geq 2^tN^{1-1/s}-1/q! \geq N^{1-1/s}.$$
We obtain a set $U_1$ of size at least $N^{1-1/s}$ such that every subset of $U_1$ of size $q$ has at least $m$
common neighbors in $G$.

Choose a random subset $T \subset U_1$ consisting of $q-r$ (not necessarily distinct) uniformly chosen vertices of
$U_1$. Since $s\geq 2$ we have that $q-r=\frac{7}{4}rs-r \geq \frac{5}{4}rs$. Let $U_2$ be the set of common neighbors
of $T$. The probability that $U_2$ contains a subset of size $r$ with
at most $m$ common neighbors in $U_1$ is at most
$${N \choose r}\left(\frac{m}{|U_1|}\right)^{q-r} \leq \frac{N^r}{r!}N^{-0.8(q-r)/s}
\leq 1/r!<1,$$
where we used that $m=N^{1-1.8/s}$ and $|U_1| \geq N^{1-1/s}$.

Therefore there is a choice of $T$ such that every subset of $U_2$ of size $r$ has at least $m$
common neighbors in $U_1$. Consider now an arbitrary subset $S$ of $U_1$ of size at most $r$. Since $S \cup T$ is
a subset of $U_1$ of size at most $q$, this set has at least $m$ common neighbors in $G$.
Observe, crucially, that by definition of $U_2$ all
common neighbors of $T$ in $G$ lie in $U_2$. Thus it follows that $N(S \cup T) \subset N(T) \subset U_2$.
Hence $S$ has at least $m$ common neighbors in $U_2$ and the statement is proved.
\end{proof}

From the above two lemmas, we get immediately the following corollary.

\begin{corollary} \label{cordeg} If $r,s \geq 2$ and $G$ is a graph with $N$ vertices and at least $N^{2-1/(s^3r)}$
edges, then $G$ contains every $r$-degenerate bipartite graph with at most $N^{1-1.8/s}$ vertices. \end{corollary}

\subsection{Applications}
We present three quick applications of Corollary \ref{cordeg} to extremal problems for degenerate graphs.

\vspace{0.2cm}
\noindent {\bf Ramsey numbers of degenerate graphs}
\vspace{0.2cm}

\noindent
As we already mentioned, Ramsey numbers of sparse graphs play a central role in
graph Ramsey theory. In 1975 Burr and Erd\H{o}s \cite{BuEr} conjectured that for each $r$,
there is a constant $c(r)$ such that $r(H)
\leq c(r)n$ for every $r$-degenerate graph $H$ on $n$ vertices.
This conjecture is a substantial generalization of the bounded degree case which we discussed in detail in Section 6.
It is a well-known and difficult problem and progress on this question has only been made recently.

In addition to bounded degree graphs, the Burr-Erd\H{o}s conjecture has also been settled for several special
cases where the maximum degree is unbounded. Chen and
Schelp \cite{ChSc} proved a result which implies that planar graphs have linear Ramsey numbers. This was extended by
R\"odl and Thomas \cite{RoTh} to graphs with no $K_r$-subdivision.
Random graphs provide another interesting and large collection of degenerate graphs.
Let $G(n,p)$ denote the random graph on $n$ vertices in which each edge appears with probability $p$ independently
of all the other edges. It is easy to show that the random graph $G(n,p)$ with $p=d/n$ and constant $d$ with high
probability (w.h.p.) has bounded degeneracy and maximum degree
$\Theta(\log n/\log \log n)$. Recently, the authors \cite{FoSu3} showed that w.h.p.
such graphs also have linear Ramsey number. In some sense this result says that
the Burr-Erd\H{o}s conjecture holds for typical degenerate graphs.

Kostochka and R\"odl \cite{KoRo1} were the first to prove a
polynomial upper bound on the Ramsey numbers of general $r$-degenerate
graphs. They showed that $r(H) \leq c_rn^2$ for every $r$-degenerate
graph $H$ with $n$ vertices. A nearly linear bound $r(H)
\leq c_rn^{1+\epsilon}$ for any fixed $\epsilon>0$ was obtained by Kostochka and Sudakov \cite{KoSu}.
The best currently known bound for this problem was obtained in \cite{FoSu3} where it was proved
that $r(H) \leq 2^{c_r\sqrt{\log n}}n$.

Here we prove a nearly linear upper bound for degenerate bipartite graphs.

\begin{theorem} The Ramsey number of every $r$-degenerate bipartite graph $H$ with $n$
vertices, $n$ sufficiently large, satisfies $$r(H) \leq 2^{8r^{1/3}(\log n)^{2/3}}n.$$ \end{theorem}

\begin{proof} In every $2$-coloring of the edges of the complete graph $K_N$, one of the color classes contains at
least half of the edges. Let $N=2^{8r^{1/3}(\log n)^{2/3}}n$ and let $s=\frac{1}{2}(r^{-1}\log n)^{1/3}$.
Then $N^{2-1/(s^3r)} \leq \frac{1}{2}{N \choose 2}$ and $N^{1-1.8/s}\geq n$. By Corollary \ref{cordeg}, the majority color contains a copy of $H$.
\end{proof}

\vspace{0.2cm}
\noindent {\bf Tur\'an numbers of degenerate bipartite graphs}
\vspace{0.2cm}

\noindent
Recall that the Tur\'an number $\textrm{ex}(n,H)$ is the maximum number of edges in a graph
on $n$
vertices that contains no copy of $H$. The asymptotic behavior of these numbers is well known for graphs of chromatic number at least
$3$. For bipartite graphs, the situation is considerably more complicated. There are relatively few bipartite graphs $H$ for which the
order of magnitude of $\textrm{ex}(n,H)$ is known. It is not even clear what parameter of a bipartite graph $H$ should determine the
asymptotic behavior of $\textrm{ex}(n,H)$. Erd\H{o}s \cite{Er1} conjectured in 1967 that $\textrm{ex}(n,H) = O(n^{2-1/r})$ for every
$r$-degenerate bipartite graph $H$. The only progress on this conjecture was made recently by Alon, Krivelevich, and Sudakov
\cite{AlKrSu}, who proved that $\textrm{ex}(n,H) \leq
h^{1/2r}n^{2-4/r}$ for graph $H$ with $h$ vertices.
Substituting $s=2$ in Corollary \ref{cordeg} gives the following, slightly weaker bound.

\begin{theorem}
Let $H$ be an $r$-degenerate bipartite graph on $h$ vertices and let  $n>h^{10}$. Then
$$ex(n,H) < n^{2-\frac{1}{8r}}.$$
\end{theorem}

\vspace{0.2cm}
\noindent {\bf Ramsey numbers of graphs with few edges}
\vspace{0.2cm}

\noindent
One of the basic results in Ramsey theory, mentioned at the beginning of Section 6, says that
$r(H) \leq 2^{O(\sqrt{m})}$ for the complete graph $H$ with $m$ {\it edges}.
Erd\H{o}s \cite{Er2} conjectured in the early 80's that a similar bound holds for every graph $H$ with $m$ edges and no
isolated vertices. Alon, Krivelevich, and Sudakov \cite{AlKrSu} verified this conjecture for bipartite graphs and also showed that $r(H) \leq 2^{O(\sqrt{m} \log m)}$ for every graph $H$ with $m$ edges. Recently, Sudakov \cite{Su2} proved the conjecture. Here we present the short proof of the result for bipartite graphs.

\begin{theorem} 
Let $H$ be a bipartite graph with $m$ edges and no isolated vertices. Then $r(H) \leq 2^{16\sqrt{m}+1}$.
\end{theorem}
\begin{proof}
First we prove that $H$ is $\sqrt{m}$-degenerate. If not, $H$ has a subgraph $H'$ with minimum degree larger than
$\sqrt{m}$. Let $(U,W)$ be the bipartition of $H'$. Then the size of $U$ is larger than $\sqrt{m}$ since every vertex in $W$ has
has more than $\sqrt{m}$ neighbors in $U$. Therefore the number of edges in $H'$ (and hence in $H$) is at least $\sum_{v\in
U}d(v)>\sqrt{m}|U|>m$, a contradiction.

Let $N=2^{16\sqrt{m}+1}$ and consider a $2$-coloring of the edges of $K_N$. Clearly, at least
$\frac{1}{2}{N \choose 2} \geq N^{2-\frac{1}{8\sqrt{m}}}$ edges have the same color. These edges form a monochromatic graph which
satisfies Corollary \ref{cordeg} with $r=\sqrt{m}$ and $s=2$. Thus this graph contains every
$\sqrt{m}$-degenerate bipartite graph on at most $N^{1/10}>2^{1.6\sqrt{m}} > 2m$ vertices.
In particular, it contains $H$ which has at most $m$  edges and therefore at most $2m$ vertices.
\end{proof}

\section{Embedding $1$-subdivided graphs}
\label{1subdivided}
Recall that a $1$-subdivision of a graph $H$ is a graph formed by
replacing edges of $H$ with internally vertex disjoint paths of length $2$.
This is a special case of a more general notion of topological copy of a graph, which plays
an important role in graph theory.
In Section \ref{first1subdividesect}, we discussed a proof of the old conjecture of Erd\H{o}s that for each $\epsilon>0$ there is
$\delta>0$ such that every graph with $n$ vertices and at least $\epsilon n^2$ edges contains the $1$-subdivision of a complete graph
of order $\delta n^{1/2}$. In this section, we describe two extensions of this result, each requiring a new variation of the
basic dependent random choice approach.

The first extension gives the right dependence of $\delta$ on $\epsilon$ for the conjecture of Erd\H{o}s. The proof in Section
\ref{first1subdividesect} shows that we may take $\delta=\epsilon^{3/2}$. We present the proof of Alon, Krivelevich, and Sudakov
\cite{AlKrSu} that this can be improved to $\delta=\epsilon$. On the other hand, the following simple probabilistic argument shows that
the power of $\epsilon$ in this results cannot be further improved. Suppose that we can prove $\delta=\epsilon^{1-1/t}$ for some $t>0$ and consider a random
graph $G(n,p)$ with $p=n^{-1/2-1/(2t)}$. With high probability this graph has $\Omega(n^{3/2-1/(2t)})$ edges and contains no $1$-subdivision of the clique of order $2t+2$. Indeed, such a subdivision has $v=(2t+3)(t+1)$ vertices and $e=(2t+2)(2t+1)$ edges. Therefore, the expected number of copies of such a subdivision is at most $n^vp^e=o(1)$. Then it is easy to check that taking $\epsilon$ to be of order $n^{-1/2-1/(2t)}$ gives a contradiction. Note that a clique of order $O(n^{1/2})$ has $O(n)$ edges. So one can naturally ask whether under the same conditions one can find a $1$-subdivision of every graph with at most $\delta n$ edges, not just of a clique.
In \cite{FoSu5} we show that this is indeed the case and for each $\epsilon>0$ there is a $\delta>0$ such that every graph with $n$ vertices and at least $\epsilon n^2$ edges contains the $1$-subdivision of every graph $\Gamma$ with at most $\delta n$ edges and vertices.

\subsection{A tight bound for $1$-subdivisions of complete graphs}

The goal of this subsection is to prove the following result.

\begin{theorem}\label{main1subdividecomplete}
If $G$ is a graph with $n$ vertices and $\epsilon n^2$ edges, then $G$ contains the $1$-subdivision of the complete graph with $\epsilon n^{1/2}$ vertices.
\end{theorem}

This theorem follows immediately from the lemma below. This lemma uses dependent random choice to find
a large set $U$ of vertices such that for each $i$, $1 \leq i \leq {|U| \choose 2}$, there are less than $i$ pairs
of vertices in $U$ with fewer than $i$ common neighbors outside $U$. Indeed, suppose we have found, in the graph $G$, a vertex subset $U$ with $|U|=k$ such that for each
$i$, $1 \leq i \leq {k \choose 2}$,  there are less than $i$ pairs of vertices in $U$ with fewer than $i$ common neighbors
in $G\setminus U$. Label all the pairs $S_1,\ldots,S_{{k \choose 2}}$ of vertices of $U$ in non-decreasing order of the size of $|N(S_i) \setminus U|$.
Note that for all $i$ we have by our assumption that $|N(S_i) \setminus U| \geq i$.
We find distinct
vertices $v_1,\ldots,v_{{k \choose 2}}$ such that $v_i \in N(S_i) \setminus U$. These vertices together with $U$ form a copy of the $1$-subdivision of the complete graph of order $k$
in $G$, where $U$ corresponds to the vertices of the complete
graph, and each pair $S_i$ is connected by path of length 2 through $v_i$. We construct the sequence $v_1,\ldots,v_{{k \choose 2}}$ of vertices one by one.
Suppose we have found $v_1,\ldots,v_{i-1}$, we can let
$v_i$ be any vertex in $N(S_i) \setminus U$ other than $v_1,\ldots,v_{i-1}$. Such a vertex $v_i$ exists since $|N(S_i) \setminus U| \geq i$.
Thus to finish the proof of Theorem \ref{main1subdividecomplete} we only need to prove the following.

\begin{lemma}
Let $G=(V,E)$ be a graph with $n$ vertices and $\epsilon n^2$ edges, and let $k=\epsilon n^{1/2}$. Then $G$ contains a subset $U \subset V$ with $|U|=k$ such that for each $i$, $1 \leq i \leq
{k \choose
2}$, there are less than $i$ pairs of vertices in $U$ with fewer than $i$ common neighbors in $G \setminus U$.
\end{lemma}
\begin{proof}
Partition $V=V_1 \cup V_2$ such that $|V_1|=|V_2|=n/2$ such that at least half of the edges of $G$ cross $V_1$ and $V_2$.
Let $G_1$ be the bipartite subgraph of $G$ consisting of those edges that cross $V_1$ and $V_2$.
Without loss of generality, assume that $\sum_{v \in V_1} |N_{G_1}(v)|^2 \leq \sum_{v \in V_2} |N_{G_1}(v)|^2$.

Pick a pair $T$ of vertices of $V_1$ uniformly at random with
repetition. Set $A=N_{G_1}(T) \subseteq V_2$, and let $X$ denote the cardinality of $A$.
By linearity of expectation,
$$\mathbb{E}[X]=\sum_{v \in V_2}\left(\frac{|N_{G_1}(v)|}{n/2}\right)^2
=4n^{-2}\sum_{v\in V_2} |N_{G_1}(v)|^2 \geq 2n^{-1}\left(\frac{\sum_{v\in
V_2} |N_{G_1}(v)|}{n/2} \right)^2 \geq 2\epsilon^2n,$$
where the first inequality is by convexity of the function $f(z)=z^2$.

Define the weight $w(S)$ of a subset $S\subset V_2$ by $w(S)=\frac{1}{|N_{G_1}(S)|}$. Let $Y$ be the random variable which sums the weight of all pairs
$S$ of vertices in $A$.
We have
\begin{eqnarray*}
\mathbb{E}[Y] & = & \sum_{S \subset V_2, |S|=2} w(S){\mathbb{P}}(S \subset A)=\sum_{S \subset V_2, |S|=2} w(S)\left(\frac{|N_{G_1}(S)|}{n/2} \right)^2  =
4n^{-2}\sum_{S \subset V_2, |S|=2} |N_{G_1}(S)| \\ & = & 4n^{-2}\sum_{v\in V_1} {|N_{G_1}(v)| \choose 2} < 2n^{-2}\sum_{v\in V_1} |N_{G_1}(v)|^2
 \leq   2n^{-2}\sum_{v\in V_2} |N_{G_1}(v)|^2 =\mathbb{E}[X]/2
\end{eqnarray*}

\noindent
This inequality with linearity of expectation implies $\mathbb{E}[X-\mathbb{E}[X]/2 - Y] >0$. Hence, there is a choice of $T$ such that the
corresponding set $A$ satisfies
$X>\mathbb{E}[X]/2 \geq \epsilon^2 n$ and $X>Y$.

Let $U$ be a random subset of $A$ of size exactly $k$ and $Y_1$ be the random variable which sums the weight of all pairs $S$ of vertices in $U$.
We have
\begin{eqnarray*}
\mathbb{E}[Y_1]=\frac{{k \choose 2}}{{X \choose 2}}Y < \left(k/X\right)^2 X = k^2/X < 1.
\end{eqnarray*}

\noindent
This implies that there is a particular subset $U$ of size $k$ such that $Y_1 < 1$. In the bipartite graph $G_1$, for each $i$, $1 \leq
i \leq {k \choose 2}$, there are less than $i$ pairs of vertices in $U$ with fewer than $i$ common neighbors. Indeed, otherwise
$Y_1=\sum_{S \subset U, |S|=2} \frac{1}{|N_{G_1}(S)|}$
is at least $i\frac{1}{i} =1$. Hence, in $G$, there are less than $i$ pairs of vertices in $U$ with fewer than $i$ common neighbors
in $G\setminus U$. \end{proof}

\subsection{$1$-subdivision of a general graph}
\label{section1subdividedgeneral}

We prove the following result on embedding the $1$-subdivision of a general graph in a dense graph.

\begin{theorem}\label{subdivided1}
Let $\Gamma$ be a graph with at most $n$ edges and vertices and let
$G$ be a graph with $N$ vertices and $\epsilon N^2$ edges such that
$N \geq 128\epsilon^{-3}n$. Then $G$ contains the $1$-subdivision of
$\Gamma$.
\end{theorem}

The proof uses repeated application of dependent random choice to find a
{\it sequence of nested subsets} $A_0 \supset A_1 \supset
\ldots$ such that $A_i$ is sufficiently large and the fraction of pairs in $A_i$ with small common neighborhood drops
significantly with $i$. This will be enough to embed the $1$-subdivision of $\Gamma$.
Note that the $1$-subdivision of a graph $\Gamma$ is a bipartite
graph whose first part contains the vertices of $\Gamma$ and whose
second part contains the vertices which were used to subdivide the
edges of $\Gamma$. Furthermore, the vertices in the second part have
degree two. Therefore, Theorem \ref{subdivided1} follows from Theorem \ref{subdividedthm1} below. The {\it codegree} of a pair of vertices in a graph is the number of their common neighbors.

\begin{lemma}\label{subdividedlem}
If $G$ is a graph with $N \geq 128\epsilon^{-3}n$ vertices and $V_1$ is the set of vertices with degree at least $\epsilon N/2$, then there are nested vertex subsets $V_1=A_0 \supset A_1 \supset \ldots $
such that, for all $i \geq 0$, $|A_{i+1}| \geq \frac{\epsilon}{8}|A_{i}|$ and each vertex in $A_i$ has codegree at least $n$ with all but at most $(\epsilon/8)^i|A_i|$ vertices in $A_i$.
\end{lemma}

\begin{proof}
Having already picked $A_0,\ldots,A_{i-1}$ satisfying the desired properties, we show how to pick $A_i$. Let $w$ be a vertex chosen uniformly at random. Let $A$ denote the set of neighbors of $w$ in $A_{i-1}$, and $X$ be
the random variable denoting the cardinality of $A$. Since every
vertex in $V_1$ has degree at least $\epsilon N/2$,
$$\mathbb{E}[X]=\sum_{v \in A_{i-1}}\frac{|N(v)|}{N}\geq \frac{\epsilon}{2}|A_{i-1}|.$$

Let $Y$ be the random variable counting the number of pairs in $A$
with fewer than $n$ common neighbors. Notice that the
probability that a pair $R$ of vertices of $A_{i-1}$ is in $A$ is $\frac{|N(R)|}{N}$. Let $c_i=(\epsilon/8)^i$. Recall that $E_{i-1}$ is the set
of all pairs $R$ in $A_{i-1}$ with $|N(R)|< n$ and $|E_{i-1}| \leq c_{i-1}|A_{i-1}|^2/2$.
Therefore,
$$\mathbb{E}[Y] < \frac{n}{N}|E_{i-1}| \leq \frac{n}{N}\frac{c_{i-1}}{2}|A_{i-1}|^2.$$
By convexity, $\mathbb{E}[X^2] \geq \mathbb{E}[X]^2$. Thus, using
linearity of expectation, we obtain
 $$\mathbb{E}\left[X^2-\frac{\mathbb{E}[X]^2}{2\mathbb{E}[Y]}\, Y-\mathbb{E}[X]^2/2\right] \geq 0.$$
Therefore, there is a choice of $w$ such that this expression is
nonnegative. Then
$$X^2 \geq \frac{1}{2}\mathbb{E}[X]^2 \geq \frac{\epsilon^2}{8}|A_{i-1}|^2$$ and
since $N \geq 128\epsilon^{-3}n$,
$$Y \leq 2\frac{X^2}{\mathbb{E}[X]^2}\mathbb{E}[Y] \leq 4\epsilon^{-2} c_{i-1} \frac{n}{N}X^2
\leq \frac{\epsilon}{16} c_{i-1}\frac{X^2}{2}.$$ From the first
inequality, we have $|A|=X \geq \frac{\epsilon}{4}|A_{i-1}|$ and the
second inequality states that the number of pairs of vertices in $A$ with codegree less than $n$ is at most
$\frac{\epsilon}{16} c_{i-1}\frac{|A|^2}{2}$. If $A$
contains a vertex that has codegree less than $n$ with more than $\epsilon c_{i-1}|A|/16$ other vertices of $A$, then
delete it, and continue this process until there is no remaining vertex with codegree less than $n$ with more than $\epsilon c_{i-1}|A|/16$ other remaining vertices. Let $A_i$ denote the set of remaining vertices. The number of deleted vertices is at most $\left(\frac{\epsilon}{16} c_{i-1}\frac{|A|^2}{2}\right)/\left(\epsilon c_{i-1}|A|/16\right)=|A|/2$. Hence, $|A_i| \geq |A|/2 \geq \frac{\epsilon}{8}|A_{i-1}|$ and every vertex in $A_i$ has codegree at least $n$ with all but at most $\epsilon c_{i-1}|A|/16 \leq \epsilon c_{i-1}|A_i|/8 =c_i|A_i|$ vertices of $A_i$. By induction on $i$, this completes the proof.
\end{proof}

\begin{theorem}\label{subdividedthm1}
If $G$ is a graph with $N \geq 128\epsilon^{-3}n$ vertices and $\epsilon N^2$ edges, then $G$ contains every bipartite graph $H=(U_1,U_2;F)$ with $n$ vertices such that every vertex in $U_2$ has degree $2$.
\end{theorem}
\begin{proof}
The set $V_1$ of vertices of $G$ of degree at least $\epsilon N/2$ satisfies $|V_1| > \epsilon^{1/2} N$. Indeed, the number of edges of $G$ containing a vertex not in $V_1$ is at most $N \cdot \epsilon
N/2$. Hence, the number of edges with both vertices in $V_1$ is at least $\epsilon N^2/2$ and at most ${|V_1| \choose 2}$, and it follows $|V_1| > \epsilon^{1/2}N$. Applying Lemma \ref{subdividedlem},
there are nested vertex subsets $V_1=A_0 \supset A_1 \supset \ldots $ such that for all $i \geq 0$, $|A_{i+1}| \geq \frac{\epsilon}{8}|A_{i}|$ and each vertex in $A_i$ has codegree at least $n$ with
all but at most $(\epsilon/8)^i|A_i|$ vertices in $A_i$.

Let $H'$ be the graph with vertex set $U_1$ such that two vertices in $U_1$ are adjacent in $H'$ if they have a common neighbor in $U_2$ in graph $H$. If we find an embedding $f:U_1 \rightarrow V_1$ such
that for each edge $(v,w)$ of $H'$, $f(v)$ and $f(w)$ have codegree at least $n$ in $G$, then we can extend $f$ to an embedding of $H$ as a subgraph of $G$. To see this, we embed the vertices of $U_2$
one by one. If the current vertex to embed is $u \in U_2$, and $(v,w)$ is the pair of neighbors of $u$ in $U_1$, then $(v,w)$ is an edge of $H'$ and hence $f(v)$ and $f(w)$ have at least $n$ common
neighbors. As the total number of vertices of $H$ embedded so far is less than $n$, one of the common neighbors of $f(v)$ and $f(w)$ is still unoccupied and can be used to embed $u$. Thus it is enough to
find an embedding $f:U_1 \rightarrow V_1$ with the desired property.

Label the vertices $\{v_1,\ldots,v_{|U_1|}\}$ of $H'$ in non-increasing
order of their degree. Since $H'$ has at most $n$ edges, the degree
of $v_i$ is at most $2n/i$. We will embed the vertices of $H'$ in the order of their index $i$.
Let $c_j=(\frac{\epsilon}{8})^j$. The
vertex $v_i$ will be embedded in $A_j$ where $j$ is the least
positive integer such that $c_j \leq \frac{i}{4n}$. Note that, by definition,
$c_{j-1} = (\frac{\epsilon}{8})^{-1}c_j \geq \frac{i}{4n}$.
Since $N \geq 128\epsilon^{-3}n$, then
$$|A_j| \geq c_j \, |A_0| \geq  c_j \, \epsilon^{1/2} N \geq
\frac{\epsilon}{8} \, \frac{i}{4n} \, \epsilon^{1/2} N \geq
2i.$$ Assume we have already embedded all vertices $v_k$ with $k<i$ and we want to embed $v_i$. Let $N^-(v_i)$ be the set of vertices
$v_k$ with $k<i$ that are adjacent to $v_i$ in $H'$. Each vertex in $A_j$ has codegree at least $n$ with all but at most $c_j|A_j| \leq
\frac{i}{4n}|A_j|$ other vertices in $A_j$. Since $v_i$ has degree at most $\frac{2n}{i}$ in
$H'$, at least $|A_j|-\frac{2n}{i}\cdot \frac{i}{4n}|A_j| =
|A_j|/2$ vertices of $A_j$ have codegree at least $n$ with every vertex in $f(N^-(v_i))$.
Since also $|A_j|/2 \geq i$, there is a vertex in
$A_j \setminus f(\{v_1,\ldots,v_{i-1}\})$ that has codegree at least $n$ with every vertex in $f(N^-(v_i))$.
Use this vertex to embed $v_i$ and continue. This gives the desired embedding $f$, completing the
proof.
\end{proof}

\section{Graphs whose edges are in few triangles}

In this section, we discuss an application of dependent random
choice to dense graphs in which each pair of adjacent vertices has few common neighbors.
It was shown in \cite{Su1} that every such graph $G$ contains a
large induced subgraph which is sparse.
This follows from the simple observation that the expected cardinality of the common neighborhood $U$ of a small random subset of vertices of $G$ is large, while, since every edge is in few triangles, the expected
number of edges in $U$ is small. We also use this
lemma to establish two Ramsey-type results.

\begin{lemma}\label{sparseinduced} Let $t \geq 2$ and $G=(V,E)$ be a
graph with $n$ vertices and average degree $d$ such that every pair of
adjacent vertices of $G$ has at most $a$ common neighbors. Then $G$
contains an induced subgraph with at least $\frac{d^t}{2n^{t-1}}$
vertices and average degree at most $\frac{2a^t}{d^{t-1}}$. \end{lemma}

\begin{proof} Let $T$ be a subset of $t$ random vertices, chosen
uniformly with repetitions. Set $U=N(T)$, and let $X$ denote the
cardinality of $U$. By linearity of expectation and by convexity of
$f(z)=z^t$, $$\mathbb{E}[X]=\sum_{v \in
V}\left(\frac{|N(v)|}{n}\right)^t= n^{-t}\sum_{v \in V}|N(v)|^t \geq
n^{1-t}\left(\frac{\sum_{v \in V}|N(v)|}{N}\right)^t =d^tn^{1-t}.$$

Let $Y$ denote the random variable counting the number of edges in $U$.
Since, for every edge $e$, its vertices have at most $a$ common
neighbors, the probability that $e$ belongs to $U$ is at most $(a/n)^t$.
Therefore, $$\mathbb{E}[Y] \leq |E|(a/n)^t = (dn/2)(a/n)^t =
da^tn^{1-t}/2.$$ In particular, since $Y$ is nonnegative, in the case
$a=0$ we get $Y$ is identically $0$.

In the case $a=0$, there is a choice of $T$ such that $|U|=X \geq
\mathbb{E}[X] =d^tn^{1-t}$ and the number $Y$ of edges in $U$ is $0$.
Otherwise, let $Z=X-\frac{d^{t-1}}{a^t}Y-\frac{d^t}{2n^{t-1}}$. By
linearity of expectation, $\mathbb{E}[Z] \geq 0$ and thus there is
a choice of $T$ such that $Z \geq 0$. This implies $X \geq
\frac{d^t}{2n^{t-1}}$ and $X \geq \frac{d^{t-1}}{a^t}Y$. Hence the subgraph of
$G$ induced by the set $U$ has $X \geq \frac{d^t}{2n^{t-1}}$ vertices
and average degree $2Y/X \leq \frac{2a^{t}}{d^{t-1}}$. \end{proof}

\subsection{Applications}
We present two quick applications of Lemma \ref{sparseinduced} to
Ramsey-type problems. The Ramsey number $r(G,H)$ is the minimum $N$ such
that every red-blue edge-coloring of the complete graph $K_N$ contains a
red copy of $G$ or a blue copy of $H$. The classical Ramsey numbers of
the complete graphs are denoted by $r(s,t)=r(K_s,K_t)$.

\vspace{0.2cm}
\noindent {\bf $K_k$-free subgraphs of $K_s$-free graphs}
\vspace{0.2cm}

\noindent
A more general function than $r(s,t)$ was first considered almost fifty years ago in two papers of Erd\H{o}s with
Gallai \cite{ErGa} and with Rogers \cite{ErR}. For a graph $G$, let $f_k(G)$ be the maximum cardinality of
a subset of vertices of $G$ that contains no $K_k$. For $2 \leq k < s
\leq n$, let $f_{k,s}(n)$ denote the minimum of $f_k(G)$ over all
$K_s$-free graphs $G$ on $n$ vertices. Note that the Ramsey number $r(s,t)$ is the minimum $n$ such that $f_{2,s}(n) \geq t$. 
Thus, the problem of determining $f_{k,s}(n)$ extends that of determining Ramsey numbers.

Erd\H{o}s and Rogers \cite{ErR} started the investigation of this function for
fixed $s$, $k=s-1$ and $n$ tending to infinity. They proved that there is
$\epsilon(s)>0$ such that $f_{s-1,s}(n) \leq n^{1-\epsilon(s)}$. About 30
years later, Bollob\'as and Hind improved the upper bound and gave the
first lower bound, $f_{k,s}(n) \geq n^{1/(s-k+1)}$. The upper bound on $f_{k,s}(n)$
was subsequently improved by Krivelevich \cite{Kr1} and most recently by Dudek and R\"odl \cite{DuRo}.
Alon and Krivelevich \cite{AlKr} gave explicit constructions of $K_s$-free graphs without large
$K_k$-free subgraphs. The best known lower bound for this problem was obtained in
\cite{Su1}, using Lemma \ref{sparseinduced}. To illustrate this application,
we present a simplified proof of a slightly weaker bound in the case
$k=3$ and $s=5$.

First we need to recall the following well known bound on the
largest independent set in uniform hypergraphs. An {\it independent set} of a hypergraph is a subset of vertices
containing no edges. For a hypergraph $H$, the {\it independence number} 
$\alpha(H)$ is the size of the largest independent set in $H$. Let
$H$ be an $r$-uniform hypergraphs with $n$ vertices and $m \geq n/2$ edges. Let $W$ be a random subset of $H$ obtained by choosing each vertex with
probability $p=(n/(rm))^{1/(r-1)}$. Deleting one vertex from every edge in $W$, gives an independent set
with expected size $pn-p^rm =\frac{r-1}{r^{r/(r-1)}}\frac{n}{(m/n)^{1/(r-1)}}$. For graphs ($r=2$) this shows that
$\alpha(H) \geq \frac{n}{2(2m/n)}$ and for $3$-uniform hypegraphs it gives the lower bound of $\frac{2}{\sqrt{27}}\frac{n}{(m/n)^{1/2}}$
on the size of the largest independent set in $H$.

\begin{theorem} Every $K_5$-free graph $G=(V,E)$ on $n$ vertices
contains a triangle-free induced subgraph on at least $n^{5/12}/2$
vertices. \end{theorem}

\begin{proof} Let $a=n^{5/12}/2$. If $G$
contains a pair of adjacent vertices with at least $a$ common neighbors,
then this set of common neighbors is triangle-free and we are done. So
we may assume that each pair of adjacent vertices in $G$ has less than
$a$ common neighbors.

Let $d$ denote the average degree of $G$. If $d \geq n^{3/4}$, then by
Lemma \ref{sparseinduced} with $t=2$ we have that $G$ contains an
induced subgraph on at least $d^2/2n \geq n^{1/2}/2$ vertices with
average degree at most $2a^2/d \leq n^{1/12}/2$. By the above discussion,
this induced subgraph of $G$ contains an independent set of size at least
$$\frac{n^{1/2}/2}{2 \cdot n^{1/12}/2} = n^{5/12}/2.$$

So we may suppose $d < n^{3/4}$. Let $H$ be the $3$-uniform hypergraph
with vertex set $V$ whose edges are the triangles in $G$.
Since each
edge of $G$ is in less than $a$ triangles, then the number $m$ of
triangles of $G$ (and hence the number of edges of $H$) is less than $\frac{1}{3}|E|a
=\frac{1}{3}(dn/2)a < n^{13/6}/12$. Again by the above discussion,
there is an independent set in $H$ of size at
least $$\frac{2}{3^{3/2}}\frac{n}{(m/n)^{1/2}} > n^{5/12}.$$ This
independent set in $H$ is the vertex set of a triangle-free induced
subgraph of $G$ of the desired size, which completes the proof.
\end{proof}

\vspace{0.2cm}
\noindent {\bf Book-complete graph Ramsey numbers}
\vspace{0.2cm}

\noindent The {\it book with $n$ pages} is the graph $B_n$ consisting of
$n$ triangles sharing one edge. Ramsey problems involving books and their
generalizations have been studied extensively by various
researchers (see, e.g., \cite{LiRo} and its references). One of the problems, which was investigated
by Li and Rousseau is
the Ramsey number $r(B_n,K_n)$. They show that there are constants
$c,c'$ such that $cn^3/\log^2 n \leq r(B_n,K_n) \leq c'n^3/\log n$, thus
determining this Ramsey number up to a logarithmic factor. Here, following
\cite{Su1}, we show how this upper bound can be improved by a $\log^{1/2} n$ factor by
using Lemma \ref{sparseinduced}. We will need the following well known result
(\cite{Bo1}, Lemma 12.16).

\begin{proposition} \label{Bollobas} Let $G$ be a graph on $n$ vertices
with average degree at most $d$ and let $m$ be the number of triangles
of $G$. Then $G$ contains an independent set of size at least
$\frac{2n}{39d}\left(\log d - 1/2 \log (m/n)\right)$. \end{proposition}

\begin{theorem} For all sufficiently large $n$, we have $r(B_n,K_n) \leq
800 n^3/\log^{3/2} n$. \end{theorem}

\begin{proof} Let $G$ be a graph of order $N=800 n^3/\log^{3/2} n$ not
containing $B_n$. Denote by $d$ the average degree of $G$. By
definition, every edge of $G$ is contained in less than $n$ triangles. Therefore,
the total number $m$ of triangles is less than $\frac{1}{3}|E(G)|n =
\frac{1}{3}(dN/2)n < dnN$.
If $d \leq \frac{20n^2}{\log^{1/2} n}$, then, by Proposition
\ref{Bollobas}, $G$ contains an independent set of size
$$\frac{2N}{39d}\left(\log d - 1/2 \log (m/N)\right) \geq
\frac{N}{39d}\log(d/n) \geq (40/39+o(1))n > n.$$

We may therefore assume $d > \frac{20n^2}{\log^{1/2} n}$. By Lemma
\ref{sparseinduced} with $t=2$, $G$ contains an induced subgraph with at
least $\frac{d^2}{2N} > \frac{1}{4}n\log^{1/2} n$ vertices and average
degree at most $2n^2/d < \frac{1}{10}\log^{1/2} n$. As it was mentioned earlier,
this induced subgraph of $G$ contains an independent set with at least
$\frac{\frac{1}{4}n\log^{1/2} n}{2 \cdot \frac{1}{10}\log^{1/2} n} > n$
vertices. \end{proof}

\section{More applications and concluding remarks}
The results which we discussed so far were chosen mainly to illustrate different variations of the basic technique.
There are many more applications of dependent random choice. Here we mention very briefly a few additional results whose
proofs use this approach. For more details about these applications we refer the interested reader to the original papers.

\vspace{0.1cm}
\noindent {\bf Unavoidable patterns:}\,
Ramsey's theorem guarantees a large monochromatic clique in any $2$-edge-coloring of a sufficiently large complete graph. If we are
interested in finding in such a coloring a subgraph that is not monochromatic, we must assume that each color
is sufficiently represented, e.g., that each color class has at least $\epsilon{n \choose 2}$ edges. Let $\mathcal{F}_k$
denote the family of $2$-edge-colored complete graphs on $2k$ vertices in which one color forms either a clique of order
$k$ or two disjoint cliques of order $k$. Consider a $2$-edge-coloring of $K_n$ with $n$ even in which one color forms a
clique of order $n/2$ or two disjoint cliques of order $n/2$. Clearly, these colorings have at least roughly $1/4$ of the
edges in each color, and nevertheless they basically do not contain any colored patterns except those in $\mathcal{F}_k$.
This shows that the $2$-edge-colorings in $\mathcal{F}_k$ are essentially the only types of patterns that are possibly
unavoidable in $2$-edge-colorings that are far from being monochromatic.

Generalizing the classical Ramsey problem, Bollob\'as conjectured that for each $\epsilon>0$ and $k$ there is $n(k,\epsilon)$ such that every $2$-edge-coloring of $K_n$ with $n
\geq n(k,\epsilon)$ which has at least $\epsilon {n \choose 2}$ edges in each color contains a member of $\mathcal{F}_k$. This
conjecture was confirmed by Cutler and Montagh \cite{CuMo} who proved that $n(k,\epsilon)<4^{k/\epsilon}$.
Using a simple application of dependent random choice, in \cite{FoSu2} the authors improved this bound and extended the result to tournaments.
They showed that $n(k,\epsilon)<(16/\epsilon)^{2k+1}$, which is tight up to a constant factor in the exponent for all $k$ and $\epsilon$.

\vspace{0.1cm}
\noindent {\bf Almost monochromatic $K_4$:}\,
The multicolor Ramsey number $r(t;k)$ is the minimum $n$ such that every $k$-edge-coloring of $K_n$ contains a
monochromatic $K_t$. Schur in 1916 showed that $r(3;k)$ is at least exponential in $k$ and at most a constant times $k!$.
Despite many efforts over the past century, determining whether there is a constant $c$ such that
$r(3;k) \leq c^k$ for all $k$ remains a major open problem (see, e.g., the monograph \cite{GRS90}).
In 1981, Erd\H{o}s \cite{Er81} proposed to study the following generalization of the classical Ramsey problem. Let $p$ and $q$ be
integers with $2 \leq q \leq {p \choose 2}$. A $(p,q)$-coloring of $K_n$ is an edge-coloring such that every copy of $K_p$
receives at least $q$ colors. Let $f(n,p,q)$ be the minimum number of colors in a $(p,q)$-coloring of $K_n$.
Determining the numbers $f(n,p,2)$ is the same as determining the Ramsey numbers $r(p;k)$. Indeed, since a $(p,2)$-coloring contains
no monochromatic $K_p$, we have that $f(n,p,2) \leq k$ if and only if $r(p;k)>n$.

Erd\H{o}s and Gy\'arf\'as \cite{ErGy} pointed out that $f(n,4,3)$ is one of
the most intriguing open questions among all small cases. This problem can be rephrased in terms of another more
convenient function. Let $g(k)$ be the largest $n$ for which there is a $k$-edge-coloring of $K_n$ such that every $K_4$
receives at least $3$ colors, i.e., for which $f(g(k),4,3) \leq k$.
After several results by Erd\H{o}s \cite{Er81} and Erd\H{o}s and Gy\'arf\'as \cite{ErGy}, the best known lower bound for this function was obtained by Mubayi \cite{Mu}, who
showed that $g(k) \geq 2^{c\log^2 k}$ for some absolute positive constant $c$.
Until recently, the only known upper bound was $g(k)<k^{ck}$,
which follows trivially from the multicolor Ramsey number for $K_4$. Using dependent random choice, Kostochka and Mubayi
\cite{KoMu} improved this estimate to $g(k) < (\log k)^{ck}$. Extending their approach further, the authors in \cite{FoSu4} obtained
the first exponential upper bound  $g(k) < 2^{ck}$. There is still a very large gap between the lower and upper bound for this problem, and we think
the correct growth is likely to be subexponential in $k$.

\vspace{0.1cm}
\noindent {\bf Disjoint edges in topological graphs:}\,
A {\it topological graph} is a graph drawn in the plane with vertices as points and edges as curves connecting its
endpoints and passing through no other vertex. It is {\it simple} if any two edges have at most one point in common. A {\it
thrackle} is a simple topological graph in which every two edges intersect. More than 40 years ago, Conway conjectured that
every $n$-vertex thrackle has at most $n$ edges. Although, Lov\'asz, Pach, and Szegedy \cite{LoPaSz} proved a linear
upper bound on the number of edges of a thrackle, this conjecture remains
open. On the other hand, Pach and T\'oth \cite{PaTo} constructed drawings of the complete graph in the plane in which each
pair of edges intersect at least once and at most twice, showing that simplicity condition is essential.

For dense simple topological graphs, one might expect to obtain a much stronger conclusion than in the thrackle
conjecture. Indeed, in \cite{FoSu5}, we show that for each $\epsilon>0$ there is a $\delta>0$ such that every simple
topological graph with $\epsilon n^2$ edges contains two disjoint edge subsets $E_1,E_2$, each of cardinality at least
$\delta n^2$, such that every edge in $E_1$ is disjoint from every edge in $E_2$. In the case of straight-line
drawings, this result was established earlier by Pach and Solymosi \cite{PaSo}. The proof uses dependent random choice together with some geometric
and combinatorial tools. As a corollary, it was also shown that there is an absolute constant $\alpha>0$ such that any complete
simple topological graph on $n$ vertices has $\Omega\left(\log^{1+\alpha} n \right)$ pairwise disjoint edges, improving the
earlier bound of $\Omega(\log n/\log \log n)$ proved by Pach and T\'oth \cite{PaTo}. Still, the correct bound is likely to be
$\Omega(n^{\alpha})$, and there is no known sublinear upper bound.

\vspace{0.1cm}
\noindent {\bf Sidorenko's conjecture:}\, 
A beautiful conjecture of Erd\H{o}s-Simonovits \cite{Sim} and Sidorenko \cite{Si3} states that if $H$ is a bipartite graph, then the random graph with edge
density $p$ has in expectation asymptotically the minimum number of copies of
$H$ over all graphs of the same order and edge density. This is known to be true only in several special cases, e.g., for complete bipartite graphs, trees, even cycles and, recently, for cubes. The original formulation of the conjecture by Sidorenko is in terms of graph homomorphisms. A {\it homomorphism} from a graph $H$ to a graph $G$ is a mapping $f:V(H) \rightarrow V(G)$ such that, for each edge $(u,v)$ of $H$, $(f(u),f(v))$
is an edge of $G$. Let $h_H(G)$ denote the number of homomorphisms from $H$ to
$G$. The normalized function $t_H(G)=h_H(G)/|G|^{|H|}$ is the fraction of mappings $f:V(H) \rightarrow V(G)$ which
are homomorphisms. Sidorenko's conjecture states that for every bipartite graph $H$ with $m$ edges and every graph $G$, $$t_H(G) \geq
t_{K_2}(G)^m.$$ This inequality has an equivalent analytic form which involves integrals known as Feynman integrals in quantum field theory, and has connections with Markov chains, graph limits, and Schatten-von Neumann norms. 
 
Recently, Conlon and the authors \cite{CoFoSu10a} proved that Sidorenko's conjecture holds for every bipartite graph $H$ which has a vertex complete to the other part. It is notable that dependent random choice was vital to the proof of this tight inequality. From this result, we may easily deduce an approximate version of Sidorenko's conjecture for all bipartite graphs. Define the {\it width} of a bipartite graph $H$ to be the minimum number of edges needed to be added to $H$ to obtain a bipartite graph with a vertex complete to the other part. The width of a bipartite graph with $n$
vertices is at most $n/2$. As a simple corollary, if $H$ is a bipartite graph with $m$ edges and width $w$, then $t_H(G) \geq
t_{K_2}(G)^{m+w}$ holds for every graph $G$.

\vspace{0.1cm}
\noindent {\bf Testing subgraph in directed graphs:}\,
Following Rubinfield and Sudan \cite{RuSu} who introduced the notion of property testing,
Goldreich, Goldwasser, and Ron \cite{GoGoRo} started the investigation of property testers for combinatorial
objects. A {\it property $\cal P$}  is a family of digraphs closed under isomorphism. A directed graph (digraph) $G$ with
$n$ vertices is {\it $\epsilon$-far from satisfying $\cal P$} if one must add or delete at least
$\epsilon n^2$ edges in order to turn $G$ into a digraph satisfying $\cal P$.
An {\it $\epsilon$-tester} for $\cal P$ is a
randomized algorithm, which given $n$ and  the ability to check whether there is an edge between given pair of vertices,
distinguishes with probability at least $2/3$ between the case $G$ satisfies $\cal P$ and $G$ is $\epsilon$-far
from satisfying $\cal P$. Such an $\epsilon$-tester is {\it one-sided} if, whenever $G$ satisfies $\cal P$, the $\epsilon$-tester
determines this is with probability $1$.

Let $H$ be a fixed directed graph on $h$ vertices. Alon and Shapira \cite{AlSh}, using a directed version of Szemer\'edi's
regularity lemma, proved that that there is a one-sided property tester for testing the property ${\cal P}_H$ of not containing $H$ as a subgraph whose query complexity is bounded by a function of $\epsilon$ only. As is common with applications of the regularity lemma,
the function depending only on $\epsilon$ is extremely fast growing. It is therefore interesting
to determine the digraphs $H$ for which ${\cal P}_H$ is testable in time polynomial in $1/\epsilon$. A function $\phi$ mapping the
vertices of a digraph $H$ to the vertices of a digraph $K$ is a {\it homomorphism} if $(\phi(u),\phi(v))$ is an edge of $K$ whenever $(u,v)$ is an edge of $H$. The {\it core} $K$ of $H$ is the subgraph $K$ of $H$ with the smallest
number of edges for which there is a homomorphism from $H$ to $K$. Alon and Shapira prove that there is a one-sided
property tester for ${\cal P}_H$ whose query complexity is bounded by a polynomial in $1/\epsilon$ if and only if the core of $H$
is a $2$-cycle or an oriented tree. The proof when the core of $H$ is an oriented tree uses dependent random choice.

\vspace{0.1cm}
\noindent {\bf Ramsey properties and forbidden induced subgraphs:}\,
A graph is {\it $H$-free} if it does not contain $H$ as an {\em
induced} subgraph. A basic property of large random graphs is that
they almost surely contain any fixed graph $H$ as an induced
subgraph. Conversely, there is a general belief that $H$-free graphs
are highly structured. For example, Erd\H{o}s and Hajnal \cite{ErHa}
proved that every $H$-free graph on $N$ vertices contains a
homogeneous subset (i.e., clique or independent set) of size at least $2^{c_H\sqrt{\log N}}$.
This is in striking contrast with the general case where one cannot guarantee a homogeneous subset
of size larger than logarithmic in $N$. Erd\H{o}s and Hajnal further
conjectured that this bound can be improved to $N^{c_H}$. This famous
conjecture has only been solved for few specific graphs $H$.

An interesting partial result for the general case was obtained by
Erd\H{o}s, Hajnal, and Pach \cite{ErHaPa}. They show that every
$H$-free graph $G$ with $N$ vertices or its complement $\bar G$
contains a complete bipartite graph with parts of size $N^{c_H}$.
A strengthening of this result which brings it closer to
the Erd\H{o}s-Hajnal conjecture was obtained in
\cite{FoSu5}, where the authors proved that any $H$-free graph on
$N$ vertices contains a complete bipartite graph with parts of size
$N^{c_H}$ or an independent set of size $N^{c_H}$.
To get a better understanding of the properties of $H$-free graphs,
it is also natural to ask for an {\it asymmetric} version of the
Erd\H{o}s-Hajnal result. Although it is not clear how to obtain such results from the original proof of Erd\H{o}s and Hajnal,
in \cite{FoSu5} we show that there exists $c=c_H>0$ such that for any $H$-free graph $G$ on $N$ vertices and
$n_1,n_2$ satisfying $(\log n_1)(\log n_2) \leq c\log N$, $G$ contains a clique of size $n_1$ or an independent set of size $n_2$.
The proof of both of the above mentioned results from \cite{FoSu5} use dependent
random choice together with an embedding lemma similar to the proof of Theorem \ref{maindensity}.

\vspace{0.1cm}
\noindent {\bf On a problem of Gowers:}\,
For a prime $p$, let $\mathbb{Z}_p$ denote the set of integers mod $p$ and $A$ be a subset of size $\lfloor p/2 \rfloor$.
For a random element $x \in \mathbb{Z}_p$, the expected size of $A \cap (A+x)$ is $|A|^2/p \approx p/4$. Gowers asked whether
there must be an $x \in \mathbb{Z}_p$ such that $A \cap (A+x)$ has approximately this size? This question was answered affirmatively
by Green and Konyagin \cite{GrKo}. The best current bound is due to Sanders \cite{Sa}, who showed that there is $x \in \mathbb{Z}_p$ such that $||A \cap (A+x)|-p/4|=O(p/\log^{1/3} p)$. Both these results used Fourier analysis. The first purely combinatorial proof was recently given by
Gowers \cite{Go1} using a new graph regularity lemma. This new regularity lemma is weaker than Szemeredi's regularity
lemma, but gives much better bounds. The proof of this new regularity lemma relies heavily on dependent random choice.

\vspace{0.1cm}
\noindent {\bf Induced subgraphs of prescribed size:}\,
There are a number of interesting problems like the Erd\H{o}s-Hajnal conjecture (mentioned above) which indicate that every graph $G$ which
contains no large homogeneous set is random-like. Let $q(G)$ denote the size of the largest homogeneous set in $G$ and
$u(G)$ denote the maximum integer $u$ such that $G$ contains for every integer $0 \leq y \leq u(G)$, an induced subgraph
with $y$ edges. Erd\H{o}s and McKay \cite{Er92} conjectured that for every $C$ there is a $\delta=\delta(C)>0$ such that
every graph $G$ on $n$ vertices and $q(G) \leq C\log n$ satisfies $u(G) \geq \delta n^2$. They only prove the much weaker
estimate $u(G) \geq \delta \log^2 n$.

Alon, Krivelevich, and Sudakov \cite{AlKrSu1} improved this bound considerably to $u(G) \geq n^{\delta}$.
Moreover, they conjecture that even if $q(G) \leq n/4$ (is rather large), still $u(G) =
\Omega(|E(G)|)$. This would imply the conjecture of Erd\H{o}s and McKay as it was shown by Erd\H{o}s and Szemer\'edi \cite{ErSzem}
that any graph $G$ on $n$ vertices with $q(G)=O(\log n)$ has $\Theta(n^2)$ edges. In \cite{AlKrSu1} the authors make some progress on this new conjecture, proving that as long as $q(G) \leq n/14$, then $u(G) \geq e^{c \log n}$. The proofs of both these results
use dependent random choice.

\vspace{0.1cm}
\noindent {\bf Concluding remarks:}\,
In this survey, we made an effort to provide a systematic coverage of variants and applications of dependent random choice and
we hope that the reader will find it helpful in mastering this powerful technique. Naturally due to space limitation, some additional applications of dependent random choice were left out of this paper (see, e.g., \cite{DuErRo, KrSu, NiRo, Ve}). Undoubtedly many more such results will appear in the future and will make this fascinating tool even more diverse and appealing.

Unfortunately, the current versions of dependent random choice can not be used for very sparse graphs. Indeed, there are graphs with $n$ vertices and $\Theta(n^{3/2})$ edges which have no cycles of length $4$. Every pair of vertices in such a graph has at most one common neighbor. One plausible way to adapt this technique to sparse graphs is, instead of picking the common neighborhood of a small random set of vertices, to pick the set of vertices which are close to all vertices in the small random set. One may try to show that such a set is not small, and deduce other properties of this set in order to prove results about sparse graphs.

In Lemma \ref{firstlemma}, and in other versions of dependent random choice, we pick a set of $t$ random vertices, and show that with positive probability the common neighborhood of this random set has certain desired properties. This explains why we assume $t$ is an integer. However, it seems likely that the assertion of Lemma \ref{firstlemma} also remains valid for non-integer values of $t$. Such a result would be useful for some
applications. Unfortunately, it is not clear how to extend the current techniques to prove this.

\vspace{0.1cm} \noindent {\bf Acknowledgment.}\, We would like to thank
David Conlon, Jacques Verstra\"ete, Jan Hladk\'y, and the anonymous referees for carefully reading this manuscript and helpful comments.


\begin{thebibliography}{}

\bibitem{AlDuLeRoYu}
N. Alon, R. A. Duke, H. Lefmann, V. R\"odl, R. Yuster, The
algorithmic aspects of the regularity lemma, {\it J. Algorithms}
{\bf 16} (1994), 80--109.

\bibitem{AlKr}
N. Alon and M. Krivelevich,
Constructive bounds for a Ramsey-type problem, {\it Graphs Combin.} {\bf 13} (1997), 217--225.

\bibitem{AlKrSu}
N. Alon, M. Krivelevich, B. Sudakov, Tur\'an numbers of bipartite
graphs and related Ramsey-type questions, {\it Combin. Probab.
Comput.} {\bf 12} (2003), 477--494.

\bibitem{AlKrSu1}
N. Alon, M. Krivelevich and B. Sudakov, Induced subgraphs of
prescribed size, {\it J. Graph Theory} {\bf 43} (2003), 239--251.

\bibitem{ARS} N. Alon, L. R\'onyai and T. Szab\'{o}, Norm-graphs:
variations and applications, {\it J. Combinatorial Theory
Ser. B} {\bf 76} (1999), 280--290.

\bibitem{AlSh}
N. Alon and A. Shapira, Testing subgraphs in directed graphs, {\it
Journal of Computer and System Sciences}, {\bf 69} (2004), 354--382.

\bibitem{AlSp}
N. Alon and J. H. Spencer, {\bf The probabilistic method,} 3rd ed.,
Wiley, 2008.

\bibitem{BaSz}
A. Balog and E. Szemer\'edi, A statistical theorem of set addition, {\it Combinatorica} {\bf 14} (1994), 263--268.

\bibitem{Be}
J. Beck, An upper bound for diagonal Ramsey numbers, {\it Studia
Sci. Math. Hungar} {\bf 18} (1983), 401--406.

\bibitem{Bo}
B. Bollob\'as, {\bf Extremal Graph Theory,} Academic Press, 1978.

\bibitem{Bo1}
B. Bollob\'as, {\bf Random graphs,} 2nd edition, Cambridge Studies in Advanced Mathematics 73, Cambridgue University Press, Cambridge, 2001.

\bibitem{BoEr}
B. Bollob\'as and P. Erd\H{o}s, On a Ramsey-Tur\'an type problem, {\it J. Combinatorial Theory Ser. B} {\bf 21} (1976), 166--168.

\bibitem{BoHi}
B. Bollob\'as and H. R. Hind, Graphs without large triangle free subgraphs, {\it Discrete Math.} {\bf 87}(2) (1991), 119--131.

\bibitem{BoTh}
B. Bollob\'as and A. Thomason, Proof of a conjecture of Mader,
Erd\H{o}s and Hajnal on topological complete subgraphs, {\it
European J. Combin.} {\bf 19} (1998), 883--887.

\bibitem{BuEr}
S. A. Burr and P. Erd\H{o}s, On the magnitude of generalized Ramsey
numbers for graphs, in: {\it Infinite and Finite Sets I}, 10,
 Colloq. Math. Soc. Janos Bolyai, North-Holland,
Amsterdam, 1975, 214--240.

\bibitem{ChSc}
G. Chen and R. H. Schelp, Graphs with linearly bounded Ramsey
numbers, {\it J. Combin. Theory Ser. B} {\bf 57} (1993), 138--149.

\bibitem{ChRoSzTr}
V. Chv\'atal, V. R\"odl, E. Szemer\'edi, and W. T. Trotter, Jr., The
Ramsey number of a graph with bounded maximum degree, {\it J.
Combin. Theory Ser. B} {\bf 34} (1983), 239--243.

\bibitem{Co}
D. Conlon, Hypergraph packing and sparse bipartite Ramsey numbers, {\it Combin. Probab.
Comput.} {\bf 18} (2009), 913--923.

\bibitem{CoFoSu}
D. Conlon, J. Fox, and B. Sudakov, Ramsey numbers of sparse
hypergraphs, {\it Random Structures and Algorithms} {\bf 35} (2009), 1--14.

\bibitem{CoFoSu10a}
D. Conlon, J. Fox, and B. Sudakov, An approximate version of Sidorenko's conjecture, submitted.

\bibitem{CoFoSu10} 
D. Conlon, J. Fox, and B. Sudakov, On two problems in graph Ramsey theory, submitted. 

\bibitem{CFKO07}
O. Cooley, N. Fountoulakis, D. K\"uhn, D. Osthus, 3-uniform
hypergraphs of bounded degree have linear Ramsey numbers, {\it J.
Combin. Theory Ser. B} {\bf 98} (2008), 484--505.

\bibitem{CFKO072}
O. Cooley, N. Fountoulakis, D. K\"uhn, D. Osthus, Embeddings
and Ramsey numbers of sparse $k$-uniform hypergraphs, {\it
Combinatorica} {\bf 29} (2009), 263--297.

\bibitem{CuMo}
J. Cutler and B. Montagh, Unavoidable subgraphs of colored graphs, {\it Discrete Math.} {\bf 308} (2008), 4396--4413.

\bibitem{DuRo}
A. Dudek and V. R\"odl,
On $K_{s}$-free subgraphs in $K_{s+k}$-free graphs and vertex Folkman numbers, Combinatorica, to appear.

\bibitem{DuErRo2}
R. A. Duke,  P. Erd\H{o}s, and V. R\"odl, More results on subgraphs
with many short cycles. Proceedings of the fifteenth Southeastern
conference on combinatorics, graph theory and computing (Baton
Rouge, La., 1984), {\em  Congr. Numerantium} {\bf 43} (1984), 295--300.

\bibitem{DuErRo}
R. A. Duke, P. Erd\H{o}s, and V. R\"odl, On large intersecting
subfamilies of uniform setfamilies, {\it Random Structures
Algorithms} {\bf 23} (2003), 351--356.

\bibitem{Er3}
P. Erd\H{o}s, Some remarks on the theory of graphs, {\it Bull. Amer.
Math. Soc.} {\bf 53} (1947), 292--294.

\bibitem{Er1} P. Erd\H{o}s,
Some recent results on extremal problems in graph theory,
in: {\em Theory of Graphs (Rome, 1966)}
Gordon and Breach, New York, 1967, 117--123.

\bibitem{Er}
P. Erd\H{o}s, Problems and results in graph theory and combinatorial
analysis, in: {\it Graph theory and related topics} (Proc. Conf.
Waterloo, 1977), Academic Press, New York (1979), 153--163.

\bibitem{Er81}
P. Erd\H{o}s, Solved and Unsolved problems in combinatorics and combinatorial number theory, Congressus Numerantium {\bf 32} (1981), 49--62.

\bibitem{Er2} P. Erd\H{o}s,
On some problems in graph theory, combinatorial analysis and
combinatorial number theory, in: {\em Graph theory and combinatorics
(Cambridge, 1983)}, Academic Press, London, 1984, 1--17.

\bibitem{Er92}
P. Erd\H{o}s, Some of my favourite problems in various branches of combinatorics, {\it Matematiche (Catania)} {\bf 47} (1992), 231--240.

\bibitem{ErGa}
P. Erd\H{o}s and T. Gallai, On the minimal number of vertices representing the edges of a graph, {\it Publ Math Inst Hungar Acad Sci} {\bf 6} (1961), 181--203.

\bibitem{ErGy}
P. Erd\H{o}s and A. Gy\'arf\'as, A variant of the classical Ramsey problem, {\it Combinatorica} {\bf 17} (1997), 459--467.

\bibitem{ErHa}
P. Erd\H{o}s and A. Hajnal, Ramsey-type theorems, {\it Discrete
Appl. Math.} {\bf 25} (1989), 37--52.

\bibitem{ErHaPa}
P. Erd\H{o}s, A. Hajnal, and J. Pach, Ramsey-type theorem for
bipartite graphs, {\it Geombinatorics} {\bf 10} (2000), 64--68.

\bibitem{EHSSS}
P. Erd\H{o}s, A. Hajnal, M. Simonovits, V. S\'os, and  E. Szemer\'edi, Tur\'an-Ramsey theorems and simple asymptotically extremal structures, {\it Combinatorica} {\bf 13} (1993), 31--56.

\bibitem{ErR}
P. Erd\H{o}s and C. A. Rogers, The construction of certain graphs, {\it Canad J. Math}  {\bf 14} (1962),
702--707.

\bibitem{ErSz}
P. Erd\H{o}s and G. Szekeres,
\newblock A combinatorial problem in geometry,
\newblock {\it Compositio Math.} {\bf 2} (1935), 463--470.

\bibitem{ErSzem}
P. Erd\H{o}s and E. Szemer\'edi, On a Ramsey type theorem,
{\it Period. Math. Hungar.} {\bf 2} (1972), 295--299.

\bibitem{FoSu1}
J. Fox and B. Sudakov, On a problem of Duke, Erd\H{o}s, and R\"odl
on cycle-connected subgraphs, {\it J. Combinatorial Theory Ser. B} {\bf 98} (2008), 1056--1062.

\bibitem{FoSu2}
J. Fox and B. Sudakov, Unavoidable patterns, {\it J. Combinatorial Theory Ser. A} {\bf 115} (2008), 1561--1569.

\bibitem{FoSu3}
J. Fox and B. Sudakov, Two remarks on the Burr-Erd\H{o}s conjecture, {\it European J. Combinatorics} {\bf 30} (2009), 1630--1645.

\bibitem{FoSu4}
J. Fox and B. Sudakov, Ramsey-type problem for an almost
monochromatic $K_4$, {\it SIAM J. Discrete Math.} {\bf 23} (2008), 155--162.

\bibitem{FoSu5}
J. Fox and B. Sudakov, Density theorems for bipartite graphs and
related Ramsey-type results, {\it Combinatorica} {\bf 29} (2009), 153--196.

\bibitem{Fu}
Z. F\"uredi, On a Tur\'an type problem of Erd\H{o}s, {\it Combinatorica} {\bf 11} (1991), 75--79.

\bibitem{GoGoRo} O. Goldreich, S. Goldwasser, and D. Ron,
Property testing and its applications to learning and approximation, {\it Journal of the ACM}  {\bf 45} (1998), 653--750.

\bibitem{Go2}
W. T. Gowers, A new proof of Szemer\'edi's theorem for arithmetic
progressions of length four, {\it Geom. Funct. Anal.} {\bf 8}
(1998), 529--551.

\bibitem{Go1} W. T. Gowers, Bipartite graphs of approximate rank 1, preprint.

\bibitem{GrRoRu}
R. Graham, V. R\"odl, and A. Ruci\'nski, On graphs with linear
Ramsey numbers, {\it J. Graph Theory} {\bf 35} (2000), 176--192.

\bibitem{GrRoRu1}
R. Graham, V. R\"odl, and A. Ruci\'nski, On bipartite graphs with
linear Ramsey numbers, {\it Combinatorica} {\bf 21} (2001),
199--209.

\bibitem{GRS90}
{R. L. Graham, B. L. Rothschild, and J. H. Spencer,} {\bf Ramsey theory},
2nd edition, {John Wiley \& Sons}, 1990.

\bibitem{GrKo}
B. Green and S. Konyagin, On the Littlewood problem modulo a prime,
{\it Canad. J. Math.} {\bf 61} (2009), 141--164.

\bibitem{Kl}
D. J. Kleitman, On a combinatorial conjecture of Erd\H{o}s, {\it J. Combinatorial Theory} {\bf 1} (1966), 209--214.

\bibitem{KRS} J. Koll\'ar, L. R\'onyai and T. Szab\'{o},
 Norm-graphs and bipartite Tur\'an numbers, {\em Combinatorica} {\bf 16}
(1996), 399--406.

\bibitem{KoSi}
J. Koml\'os and M. Simonovits, Szemer\'edi's regularity lemma and
its applications in graph theory, in: {\it Combinatorics, Paul
Erd\H{o}s is eighty, Vol. 2} (Keszthely, 1993), 295--352, Bolyai
Soc. Math. Stud., 2, J\'anos Bolyai Math. Soc., Budapest, 1996.

\bibitem{KoSz}
J. Koml\'os and E. Szemer\'edi, Topological cliques in graphs II,
{\it Combin. Probab. Comput.} {\bf 5} (1996), 79--90.

\bibitem{KoMu}
A. Kostochka and D. Mubayi, When is an almost monochromatic $K_4$
guaranteed?, {\it Combin. Probab. Comput.} {\bf 17} (2008), 823--830.

\bibitem{KoRo}
A. Kostochka and V. R\"odl, On graphs with small Ramsey numbers,
{\it J. Graph Theory} {\bf 37} (2001), 198--204.

\bibitem{KoRo1}
A. Kostochka and V. R\"odl, On graphs with small Ramsey numbers II,
{\it Combinatorica} {\bf 24} (2004), 389--401.

\bibitem{KR06}
A. Kostochka and V. R\"odl, On Ramsey numbers of uniform hypergraphs
with given maximum degree, {\it J. Combin. Theory Ser. A} {\bf 113}
(2006), 1555--1564.

\bibitem{KoSu}
A. Kostochka and B. Sudakov, On Ramsey numbers of sparse graphs,
{\it Combin. Probab. Comput.} {\bf 12} (2003), 627--641.

\bibitem{Kr1}
M. Krivelevich, Bounding Ramsey numbers through large deviation inequalities,
{\it Random Structures \& Algorithms} {\bf 7} (1995), 145--155.

\bibitem{KrSu}
M. Krivelevich and B. Sudakov, Minors in expanding graphs,
{\it Geometric and Functional Analysis} {\bf 19} (2009), 294--331.

\bibitem{Ku}
K. Kuratowski, Sur le probl\'eme des courbes gauches en topologie, {\it Fund. Math.} {\bf 15} (1930), 271--283.

\bibitem{LiRo}
Y. Li and C. Rousseau, On book-complete Ramsey numbers, {\it J. Combinatorial Theory Ser. B} {\bf 68} (1996), 36--44.

\bibitem{LoPaSz}
L. Lov\'asz, J. Pach, and M. Szegedy, On Conway's thrackle conjecture,
{\em Discrete Comput. Geom.} {\bf 18} (1997), 369--376.

\bibitem{Mu}
D. Mubayi, Edge-coloring cliques with three colors on all $4$-cliques, {\it Combinatorica} {\bf 18} (1998), 293--296.

\bibitem{NORS07}
B. Nagle, S. Olsen, V. R\"odl, M. Schacht, On the Ramsey number
of sparse 3-graphs, {\it Graphs Combin.} {\bf 24} (2008), 205--228.

\bibitem{NiRo}
V. Nikiforov and C. C. Rousseau, Ramsey goodness and beyond,
{\it Combinatorica} {\bf 29} (2009), 227--262.

\bibitem{PaSo}
J. Pach and J. Solymosi, Crossing patterns of segments, {\it J. Combin. Theory Ser. A} {\bf 96} (2001), 316--325.

\bibitem{PaTo}
J. Pach and G. T\'oth, Disjoint edges in topological graphs, in:
{\em Combinatorial Geometry and Graph Theory} (J. Akiyama et al.,
eds.), Lecture Notes in Computer Science 3330, Springer-Verlag,
Berlin, 2005, 133--140.

\bibitem{RoTh}
V. R\"odl and R. Thomas, Arrangeability and clique subdivisions, in:
{\it The mathematics of Paul Erd\H{o}s}, II (R. Graham and J. Ne\v
set\v ril, eds.) 236--239, Algorithms Combin., 14, Springer, Berlin,
1997.

\bibitem{RuSu}
R. Rubinfield and M. Sudan, Robust characterization of polynomials with applications to program testing, {\it SIAM J. on Computing} {\bf 25} (1996), 252--271.

\bibitem{Sa}
T. Sanders, The Littlewood-Gowers problem, {\it J. Anal. Math.} {\bf 101} (2007), 123--162.

\bibitem{Sh}
L. Shi, Cube Ramsey numbers are polynomial, {\it Random Structures
\& Algorithms}, {\bf 19} (2001), 99--101.

\bibitem{Si3}
A. F. Sidorenko,
\newblock A correlation inequality for bipartite graphs,
\newblock {\it Graphs Combin.} {\bf 9} (1993), 201--204.

\bibitem{Sim}
M. Simonovits, Extremal graph problems, degenerate extremal problems
and super-saturated graphs, in: {\it Progress in graph theory} (Waterloo, Ont., 1982), Academic Press, Toronto, ON, 1984, 419--437.

\bibitem{SiSo}
M. Simonovits and V. S\'os, Ramsey-Tur\'an theory, {\it Discrete Math.} {\bf 229} (2001), 293--340.

\bibitem{Su}
B. Sudakov, A few remarks on the Ramsey-Tur\'an-type problems, {\it J.
Combinatorial Theory Ser. B} {\bf 88} (2003), 99--106.

\bibitem{Su1}
B. Sudakov, Large $K_r$-free subgraphs in $K_s$-free graphs and some
other Ramsey-type problems, {\it Random Structures \& Algorithms}
{\bf 26} (2005), 253--265.

\bibitem{Su2} 
B. Sudakov, A conjecture of Erd\H{o}s on graph Ramsey numbers, submitted.

\bibitem{SuSzVu}
B. Sudakov, E. Szemer\'edi, and V. Vu, On a question of Erd\H{o}s and
Moser, {\it Duke Mathematical Journal} {\bf 129} (2005), 129--155.

\bibitem{Sz}
E. Szemer\'edi, On graphs containing no complete subgraph with $4$ vertices (Hungarian), {\it Mat. Lapok} {\bf 23} (1972), 113--116.

\bibitem{Tu}
P. Tur\'an, On an extremal problem in graph theory (Hungarian), {\it Mat. Fiz. Lapok} {\bf 48} (1941), 436--452.

\bibitem{Ve}
J. Verstra\"ete, Product representations of polynomials, {\it European J. Combinatorics} {\bf 27} (2006), 1350--1361.

\end{thebibliography}
\end{document}